\documentclass[a4paper,fleqn, final]{cas-sc}

\usepackage{tikz}
\usepackage[numbers]{natbib}
\usepackage{xcolor}
\usepackage{array, colortbl, xcolor, booktabs}
\usepackage{pgfplots}
\usepackage{afterpage}
\usepackage{amsmath}
\usepackage{float}
\usepackage{caption}
\usepackage{booktabs}
\usepackage{multirow}
\usepackage{algorithm}
\usepackage{algorithmic}
\usepackage{graphicx}
\usepackage{caption}
\usepackage{subcaption}
\usepackage{threeparttable}
\usepackage{ulem}
\usepackage{amsthm}

\usepackage{color}
\definecolor{Yingxiang}{rgb}{1, 0, 0}
\definecolor{other}{rgb}{0, 0, 1}
\definecolor{zhangshuai}{rgb}{1, 0, 1}

\newtheorem{lemma}{Lemma}

\newtheorem{theorem}{Theorem}

\begin{document}

\ExplSyntaxOn
\cs_set:Npn \__first_footerline: {}
\ExplSyntaxOff

\let\WriteBookmarks\relax
\def\floatpagepagefraction{1}
\def\textpagefraction{.001}

\shorttitle{Neural Networks for Solving PGIEPs}    

\shortauthors{}  

\title [mode = title]{Novel Product Manifold Modeling and Orthogonality-Constrained Neural Network Solver for Parameterized Generalized Inverse Eigenvalue Problems}

\author[1]{Shuai Zhang}

\ead{zhangs216@nenu.edu.cn}

\author[1]{Xuelian Jiang}
\ead{jiangxl133@nenu.edu.cn}

\author[1]{Yingxiang Xu}
\ead{yxxu@nenu.edu.cn}
\cormark[1]
\cortext[1]{Corresponding author.}

\affiliation[1]{organization={School of Mathematics and Statistics, Northeast Normal University},
city={Changchun},
postcode={130024}, 
country={China}}

\begin{keywords}
Parameterized Generalized Inverse Eigenvalue Problem \sep  Neural Networks  \sep  Orthogonal Constraint \sep Manifold
\end{keywords}

\begin{abstract}
    A parameterized orthogonality-constrained neural network is proposed for the first time to solve the parameterized generalized inverse eigenvalue problem (PGIEP) on product manifolds, offering a new perspective to address PGIEP. 
    The key contributions are twofold. First, we construct a novel model for the PGIEP, where the optimization variables are located on the product of a Stiefel manifold and a Euclidean manifold.
    This model enables the application of optimization algorithms on the Stiefel manifold, a capability that is not achievable with existing models. 
    Additionally, the gradient Lipschitz continuity of the objective function is proved. 
    Second, a parameterized Stiefel multilayer perceptron (P-SMLP) that incorporates orthogonality constraints is proposed.
    Through hard constraints, P-SMLP enables end-to-end training without the need of alternating training between the two manifolds, providing a robust computational framework for generic PGIEPs.
    Numerical experiments demonstrate the effectiveness of the proposed method.
\end{abstract}

\maketitle
\section{Introduction}
  Denote by $A(\textbf{c})$ and $B(\textbf{c})$ the affine families defined by
 \begin{equation}\label{AB}
  A(\textbf{c}) = A_0 + \sum_{i=1}^n c_i A_i \quad \text{and} \quad B(\textbf{c}) = B_0 + \sum_{i=1}^n c_i B_i,
  \end{equation}
 where the sequences $\{A_i\}_{i=0}^n$ and $\{B_i\}_{i=0}^n$ are given $n \times n$  real matrices and $\textbf{c}= (c_1, \dots, c_n)^{\top } \in \mathbb{R}^n$ is a real vector.
  The parameterized generalized inverse eigenvalue problem (PGIEP) is defined as follows.

 \textbf{PGIEP.} For given $n$ real numbers $\lambda_1\leqslant \lambda_2 \leqslant\cdots\leqslant\lambda_n$ and $A_i, B_i \in \mathbb{R}^{n\times n}(i=0,1,\dots,n)$, find $\textbf{c}= (c_1, \dots, c_n)^{\top } \in \mathbb{R}^n$ such that $A(\textbf{c})x=\lambda_i B(\textbf{c})x$ holds for $i = 1, 2, \dots, n$.

When both $A(\mathbf{c})$ and $B(\mathbf{c})$ are symmetric, we call the corresponding PGIEP symmetric, otherwise asymmetric.
 The PGIEP represents a significant subclass of inverse problems, with extensive applications in structural design \cite{joseph1992inverse}, inverse Sturm-Liouville problems \cite{osborne2006inverse,hald1972discrete}, vibrating strings \cite{thomson2018theory}, and applied mechanics \cite{majkut2010eigenvalue}.
 The PGIEP includes as special cases the additive inverse eigenvalue problem and multiplicative inverse eigenvalue problem.
 When $B_0$ is the identity matrix and all $B_i$ are zero matrices, the PGIEP reduces to the parameterized classical inverse eigenvalue problem (PCIEP) \cite{downing1956some,friedland1977inverse,chu2005inverse}.

Several models have been established for the PGIEP and solved using iterative methods
when $B(\mathbf{c})$ is positive definite.
Dai et al. \cite{dai1997newton} employed Newton’s method to solve the symmetric PGIEP, formulated as the following nonlinear system
\begin{equation}\label{fc1}
    f(\textbf{c})=\left(\lambda_1(\textbf{c})-\lambda_1,\dots,\lambda_n(\textbf{c})-\lambda_n\right)^{\top}=\lambda(\textbf{c})-\lambda^*=\textbf{0},
    \end{equation}
where $\lambda_1,\dots,\lambda_n$ are the given eigenvalues, $\lambda^*=(\lambda_1, \dots, \lambda_n)^{\top}$, and $\lambda(\textbf{c})=(\lambda_1(\textbf{c}), \dots, \lambda_n(\textbf{c}))^{\top}$.
This formulation provides a straightforward approach to the symmetric PGIEP and has been extensively investigated in the literature \cite{downing1956some,bohte1968numerical,oliveira1972matrix,friedland1987formulation}.
However, a direct application of Newton’s method to the system (\ref{fc1}) requires solving the generalized eigenvalue problem $A(\textbf{c})x=\lambda B(\textbf{c})x$ at each iteration. 
To overcome this drawback, (inexact) Newton-like method that doesn’t explicitly depend on the eigenvalues was proposed and studied in \cite{dalvand2021newton}. 
Additionally, Dai \cite{dai1999algorithm} extended the methods for solving PCIEP \cite{li1989qr,ren1992algorithms,ren1992compute} to address PGIEP by employing a QR-like decomposition, without requiring the computation of the complete solution to the generalized eigenvalue problem $A(\textbf{c})x=\lambda B(\textbf{c})x$ at each iteration, the corresponding nonlinear system was derived as

\begin{equation}\label{fc2}
    f(\textbf{c})=\left(r_{n n}^{(1)}(\textbf{c}),\dots,
        r_{n n}^{(n)}(\textbf{c})\right)^{\top}=\textbf{0},
    \end{equation}
    where each $r_{nn}^{(i)}(\textbf{c})(i=1,\dots,n)$ is obtained from the QR-like decomposition
    \begin{equation*}
  (A(\textbf{c})-\lambda_i I) P_i(\textbf{c})=Q_i(\textbf{c}) R_i(\textbf{c})
    \quad \text{with} \quad 
    R_i(\textbf{c})=\left(\begin{array}{cc}
    R_{11}^{(i)}(\textbf{c}) & R_{12}^{(i)}(\textbf{c}) \\
    0 & r_{nn}^{(i)}(\textbf{c})
    \end{array}\right),
    \end{equation*}
  $P_i(\textbf{c})$ are permutation matrices.
  Simultaneously, Dai \cite{dai1999algorithm} generalized the formulation (\ref{fc2}) from the single eigenvalue case to accommodate multiple eigenvalues.
  Depending on different methods of matrix decomposition, the nonlinear system can be further reformulated into alternative forms.
  See, for example, the following equivalent nonlinear system  \cite{dai2015solvability}

\begin{equation}\label{fc3}
    f(\textbf{c})=\left(u_{n n}^{(1)}(\textbf{c}),\dots,
        u_{n n}^{(n)}(\textbf{c})\right)^{\top}=\textbf{0},
    \end{equation}
where $u_{nn}^{(i)}(\textbf{c}) (i=1,\dots,n)$ are  obtained from the smooth LU decomposition
\begin{equation*}
    P_l^{(i)}(\textbf{c})(A(\textbf{c})-\lambda_i B(\textbf{c})) P_r^{(i)}(\textbf{c})=L_i(\textbf{c}) U_i(\textbf{c}),
\end{equation*}
    with $
U_i(\textbf{c})=\left(\begin{array}{cc}
        U_{11}^{(i)}(\textbf{c}) & U_{12}^{(i)}(\textbf{c}) \\
        0 & u_{n n}^{(i)}(\textbf{c})
        \end{array}\right)$ 
        and 
  $P_l^{(i)}(\textbf{c}), P_r^{(i)}(\textbf{c})$ being permutation matrices.

  Dalvand et al. \cite{dalvand2020extension} extended the Cayley transform method, which was developed for solving the symmetric PCIEP \cite{bai2004inexact,shen2015inexact}, to the symmetric PGIEP.

This method offers an alternative perspective for solving PGIEPs via orthogonal matrix iteration and also avoids solving a complete generalized eigenvalue problem and performing matrix factorizations at each iteration.
It is known from \cite{dai1997newton} that if the matrix pencil $(A(\textbf{c}), B(\textbf{c}))$ has prescribed eigenvalues $\lambda_1, \dots, \lambda_n (\lambda_i\neq \lambda_j,\forall i,j=1,\dots,n)$ and corresponding eigenvectors $q_1, \dots, q_n$, then there exists an orthogonal matrix Q such that
\begin{equation}\label{X5}
Q^{\top} A(\textbf{c}) Q = \Lambda \quad \text{and} \quad Q^{\top} B(\textbf{c}) Q = I,
\end{equation}
where $Q = [q_1, \dots, q_n]$ and $\Lambda = \operatorname{diag}\{\lambda_1, \dots, \lambda_n\}$.
Using the Cayley transform, the subsequent approximation $Q^{(k+1)}$ of the eigenmatrix is updated as
\begin{equation*}
Q^{(k+1)}=Q^{(k)} e^{Z^{(k)}} \simeq Q^{(k)}\left(I+\frac{1}{2} Z^{(k)}\right)\left(I-\frac{1}{2} Z^{(k)}\right)^{-1},
\end{equation*}
where the elements of the matrix \( Z^{(k)} \) are explicitly determined by \( q_i^{(k)} \), \( A(\textbf{c}^{(k)}) \), \( B(\textbf{c}^{(k)}) \), and \( \lambda_i \).
Dalvand et al. \cite{dalvand2020extension} also employed the extended Cayley transform method to solve the modified PGIEP with multiple eigenvalues described in \cite{friedland1987formulation}. This problem involves determining vector \( \textbf{c}\) such that the matrix pencil \( (A(\mathbf{c}), B(\mathbf{c})) \) has prescribed eigenvalues \( \lambda_1 = \cdots = \lambda_t < \lambda_{t+1} < \cdots < \lambda_{n-s} \) for its \( n-s \) smallest eigenvalues. But if the matrix \( B(\mathbf{c}) \) fails to be symmetric positive definite, then Eq. (\ref{X5}) does not necessarily hold. In particular, if \( B(\textbf{c}) \) is singular, Eq. (\ref{X5}) is never satisfied, then the Cayley transform method cannot be applied to solve the PGIEP.

In general, different type of PGIEPs leads to different nonlinear systems. 
For asymmetric PGIEP, based on the work in \cite{biegler1981newton,lancaster1964algorithms}, Dai \cite{dai1998some} proposed the following  nonlinear system using matrix determinant
    \begin{equation}\label{fc5}
        f(\textbf{c})=\left(\operatorname{det}(A(\textbf{c})-\lambda_1 B(\textbf{c})),\dots,\operatorname{det}(A(\textbf{c})-\lambda_n B(\textbf{c}))\right)^{\top}=\textbf{0},
    \end{equation}
which is not applicable to the multiple eigenvalue case. Afterwards,  following the work of Xu \cite{xu1996smallest}, Dai constructed the nonlinear system 
    \begin{equation}\label{fc6}
        f(\textbf{c})=\left(\sigma_{\text{min}}(A(\textbf{c})-\lambda_1 B(\textbf{c})),\dots,\sigma_{\text{min}}(A(\textbf{c})-\lambda_n B(\textbf{c}))\right)^{\top}=\textbf{0}
    \end{equation}
    to tackle this issue, 
where $\sigma_{\text{min}}(A(\textbf{c})-\lambda_i B(\textbf{c}))$ denotes the smallest singular value of matrix $A(\textbf{c})-\lambda_i B(\textbf{c})$.
From Eqs. (\ref{fc1})–(\ref{fc3}) and (\ref{fc5})–(\ref{fc6}), it is evident that solving different types of PGIEPs requires solving distinct nonlinear systems.

Among the nonlinear systems and algorithms discussed above, the PGIEP is restricted due to the constraints imposed on $B(\textbf{c})$.
This naturally raises the question: Can a model and its corresponding algorithm be developed for solving the general PGIEP problem defined in (\ref{AB}), including cases where $B(\textbf{c})$ is singular?

From the work of Dalvand et al., it is established that employing the orthogonal matrix updating idea to solve PGIEP effectively circumvents the need to solve the generalized eigenvalue problem $A(\textbf{c})x = \lambda B(\textbf{c})x$ in each iteration.
So if the PGIEP can be transformed into an optimization problem with orthogonal constraints like the standard IEP \cite{chu1990projected,chu1991constructing,chen2011isospectral,zhao2018riemannian}, then optimization methods designed for problems with orthogonality constraints can be employed for efficient solutions. Machine learning, particularly artificial neural networks (ANNs), has demonstrated promising potential for solving optimization problems, possibly overcoming limitations inherent in traditional numerical optimization techniques.
Recently, various types of neural networks have been proposed and studied to address inverse eigenvalue problems in PDEs. 
For example, k-Nearest Neighbours and Random Forests \cite{pallikarakis2024application} were employed to address the inverse Sturm-Liouville eigenvalue problem with symmetric potentials and a radial basis function neural network was used to solve an inverse problem associated with the calculation of the Dirichlet eigenvalues of the anisotropic Laplace operator \cite{ossandon2016neural}. 
The Stiefel multilayer perceptron (SMLP) \cite{zhang2024orthogonal}, an orthogonally constrained neural network architecture, was developed to solve algebraic inverse eigenvalue problems on the Stiefel manifold.
This orthogonal constraint embedded method exhibits superior computational efficiency compared to a soft-constrained MLP and is more robust in initial guess than traditional manifold optimization algorithms.
Unfortunately, the SMLP is not applicable within the current modeling framework of the PGIEP.

In this paper, motivated by the orthogonally constrained neural network for solving the algebraic inverse eigenvalue problem, we propose a novel optimization model on product manifolds consists of Stiefel and Euclidean manifolds for the PGIEP and extend the approach in \cite{zhang2024orthogonal} to solve the PGIEP.
The most significant advantage of this novel model is that it can be solved by algorithms based on orthogonal constraints, which provides an innovative computational framework for solving the PGIEP. The main contributions of this paper are as follows:
\begin{itemize}
    \item An optimization model for the PGIEP on product manifolds is proposed, which is applicable to symmetric, asymmetric cases, as well as cases with multiple eigenvalues in the PGIEP. Moreover, it is also capable of handling cases where $B(\textbf{c})$ becomes singular, which existing models are unable to solve.
     \item Parameterized Stiefel multilayer perceptron, 
    an extension of the SMLP, is developed to solve PGIEPs. The P-SMLP method simultaneously optimizes the orthogonal matrix and the vector $\textbf{c}$, eliminating the need for alternating optimization and enabling end-to-end training while rigorously preserving orthogonality constraints.
    \item To the best of our knowledge, this paper represents the first attempt to solve the general form of the PGIEP using neural networks, providing a new perspective for addressing the PGIEP.
\end{itemize}

The rest of this paper is organized as follows. In Section \ref{2}, we introduce some matrix decomposition results, concepts related to manifolds and the MLP method.
In Section \ref{3}, we discuss the new model and P-SMLP in detail. Firstly, we model the PGIEP based on the generalized real Schur decomposition and introduce a mask matrix for structural constraints. Then, we provide a theoretical analysis of the gradient Lipschitz continuity for the model. Finally, the P-SMLP is introduced by modifying the output of the SMLP, enabling end-to-end optimization to obtain the parameter vector $\textbf{c}$.
In Section \ref{4}, several numerical experiments showcase the validity and accuracy of the proposed method. In the end, conclusions are drawn in Section \ref{5}.

\section{Preliminaries}\label{2}
For ease of reading, fundamental concepts employed in this paper are briefly introduced, including concepts from matrix decomposition, manifolds, and neural networks.

\subsection{Matrix decomposition}

\begin{theorem}{(QR decomposition)}
A real square matrix $A \in \mathbb{R}^{n \times n}$ can be decomposed as
    \begin{equation*}
    A = QR,
    \end{equation*}
    where \( Q \in \mathbb{R}^{n \times n} \) is an orthogonal matrix and \( R \in \mathbb{R}^{n \times n} \) is an upper triangular matrix.     
\end{theorem}

\begin{theorem}{(Singular value decomposition)}
    For $A \in \mathbb{R}^{n \times n}$, there exist orthogonal matrices $U \in \mathbb{R}^{n \times n}$ and $V \in \mathbb{R}^{n \times n}$ and a diagonal matrix $\Sigma=\operatorname{diag}\left(\sigma_1^*, \cdots, \sigma_n^*\right) \in \mathbb{R}^{n \times n}$ with $\sigma_1^* \geqslant \sigma_2^* \geqslant \cdots \geqslant \sigma_n^* \geqslant 0$, such that
    \begin{equation*}
    A=U \Sigma V^{\top}
    \end{equation*}
    holds, where $\sigma_i^*$ represents the singular values. The column vectors of $U=\left[{u}_1, \cdots, {u}_m\right]$ are called the left singular vectors and similarly the column vectors of $V=\left[{v}_1, \cdots, {v}_n\right]$ are the right singular vectors.
\end{theorem}

\begin{theorem}{(Generalized real Schur decomposition \cite{golub2013matrix})}\label{GRSD}
Let \( A, B \in \mathbb{R}^{n \times n} \). There exist \( n \times n \) orthogonal matrices \( Q \) and \( Z \) such that \( Q^{\top} A Z \) is a block upper triangular matrix and \( Q^{\top} B Z \) is an upper triangular matrix, i.e.,
\begin{equation*}
Q^{\top} A Z = \left(\begin{array}{cccc}
A_{11} & A_{12} & \cdots & A_{1 r} \\
& A_{22} & \cdots & A_{2 r} \\
& & \ddots & \vdots \\
& & & A_{r r}
\end{array}\right):=S, \quad \text{and} \quad Q^{\top} B Z = \left(\begin{array}{cccc}
B_{11} & B_{12} & \cdots & B_{1 r} \\
& B_{22} & \cdots & B_{2 r} \\
& & \ddots & \vdots \\
& & & B_{r r}
\end{array}\right):=T,
\end{equation*}
where \( A_{jj} \) and \( B_{jj} \) are either \( 1 \times 1 \) or \( 2 \times 2 \) matrices. 
When the generalized eigenvalues of the matrix pair \( (A, B) \) are real, all \( A_{jj} \) and \( B_{jj} \) ($j = 1,\dots, r$) are \( 1 \times 1 \) matrices, and the generalized eigenvalues of the matrix pair \( (A, B) \) are given by
\begin{equation*}
\sigma(A, B) = \left\{ \frac{s_{ii}}{t_{ii}}, \quad i = 1,\dots, n \right\},
\end{equation*}
where $s_{ii}$ and $t_{ii}$ are the diagonal elements of matrices S and T, and the corresponding eigenvalue is infinite if \( t_{ii} = 0 \).
\end{theorem}

\subsection{Manifolds}

\paragraph{Stiefel manifold.}
The Stiefel manifold, denoted as $\mathcal{S}_{n, r}$, comprises $n \times r$ rectangular matrices whose column vectors are orthonormal, i.e.,
     \begin{equation*}
     \mathcal{S}_{n, r}=\left\{Q \in \mathbb{R}^{n \times r}\mid  Q^{\top}Q=\boldsymbol{I}_r\right\},
     \end{equation*}
where it is assumed that $n \geqslant r$, \(\boldsymbol{I}_r\) represents \(r \times r\) identity matrix and $Q$ is semi-orthogonal when $n > r$ and is orthogonal when $n=r$. 
Specifically, the Stiefel manifold when $n=r$ is equivalent to the orthogonal group, which is the set of all orthogonal matrices, denoted as $\mathcal{O}(n)$.

\paragraph{Euclidean manifold.}
The Euclidean space, denoted as \( \mathbb{R}^n \), is the set of all \( n \)-dimensional vectors with real components, i.e.,
\begin{equation*}
\mathbb{R}^n := \{\mathbf{x} = (x_1, x_2, \dots, x_n)^\top \}.
\end{equation*}
Noting that \( \mathbb{R}^n \) is also a manifold, named as the Euclidean manifold.

\paragraph{Product manifold.}
Let \( \mathcal{M}_1, \mathcal{M}_2, \ldots, \mathcal{M}_k \) be \( k \) manifolds. The product manifold \( \mathcal{M} \) is defined as
\begin{equation*}
\mathcal{M} = \mathcal{M}_1 \times \mathcal{M}_2 \times \cdots \times \mathcal{M}_k,
\end{equation*}
where the elements of \( \mathcal{M} \) are ordered \( k \)-tuples \( (x_1, x_2, \dots, x_k) \), with \( x_i \in \mathcal{M}_i \) for \( i = 1, 2, \dots, k \).

\subsection{Neural network.}
The MLP, characterized by its fully connected input, hidden, and output layers, represents one of the most prevalent neural network architectures in machine learning applications.
The MLP, denoted as \(\mathcal{F} = \{f: \mathbb{R}^p \rightarrow \mathbb{R}^d\}\), is represented by
\begin{equation*}
f(x) = W_{\mathcal{L}} \circ \Phi \circ W_{\mathcal{L}-1} \circ \cdots  \circ W_1 \circ \Phi \circ W_0(x),
\end{equation*}
where \(x = (x_1, \ldots, x_p)^{\top}\in \mathbb{R}^p\), \(W_i(x) = w_i x + b_i\) is an affine transformation, \(\omega_i \in \mathbb{R}^{k_{i+1} \times k_i}\) is the weight matrix, \(b_i \in \mathbb{R}^{k_{i+1}}\) is the bias vector for \(i = 0, 1, \ldots, \mathcal{L}\) with \(k_i\) representing the number of neurons in the \(i\)-th hidden layer and \(\mathcal{L}\) denoting the number of hidden layers, also known as the depth of the network;  \(\Phi\) is the activation function.

\section{Methodology}\label{3}
In this section, a novel model is presented and a P-SMLP method is designed for solving the PGIEP.
\subsection{ Modeling}
We first discuss the modeling approach for the PGIEP when $B(\mathbf{c})$ is a nonsingular matrix, and then extend the model to the scenario where $B(\mathbf{c})$ is singular.

Given $n$ real and finite numbers $\lambda_1\leqslant \lambda_2\leqslant \cdots \leqslant\lambda_n$ and $A_i, B_i \in \mathbb{R}^{n\times n}(i=0,1,\dots,n)$,

compute the generalized real Schur decomposition of $A(\textbf{c})$ and  $B(\textbf{c})$ defined in Eq.(\ref{AB}) as follows
\begin{equation*}
    Q^{\top} A(\textbf{c}) Z =  S \quad \text{and} \quad Q^{\top} B(\textbf{c}) Z = T,
\end{equation*}
where $Q,Z \in \mathcal{O}(n)$,
and $S$ and $T$ are upper triangular matrices since $\lambda_1,\dots,\lambda_n$ are real and finite, according to Theorem \ref{GRSD}.
Let $s_{ii}$ and $t_{ii}(i=1,\dots,n)$ are the diagonal elements of the matrices \( S \) and \( T \), respectively. By Theorem \ref{GRSD}, the spectrum of the generalized eigenvalue problem \( A(\textbf{c})x = \lambda B(\textbf{c})x \) is given by

\begin{equation}\label{st}
    \sigma(A(\textbf{c}),B(\textbf{c}))=\left\{\frac{s_{ii}}{t_{ii}}\mid i=1,\dots,n,\ t_{ii}\neq 0 \right\},
\end{equation}
assuming ordered eigenvalues \(\frac{s_{11}}{t_{11}} \leqslant \cdots \leqslant \frac{s_{nn}}{t_{nn}}\).
A direct result of Eq.(\ref{st}) is
\begin{equation}\label{lts}
    \lambda_it_{ii}=s_{ii},\quad i=1,\dots,n.
\end{equation}
The diagonal matrices generated by the elements \( s_{ii} \), \( t_{ii} \), and \( \lambda_it_{ii} \) are denoted as
\begin{equation*}
    \begin{aligned} 
    \operatorname{diag}\left\{s_{11}, s_{22}, \cdots, s_{n n}\right\} & =I_n \odot S=I_n \odot\left(Q^{\top} A(\textbf{c}) Z\right), \\ 
    \operatorname{diag}\left\{t_{11}, t_{22}, \cdots, t_{n n}\right\} & =I_n \odot T=I_n \odot\left(Q^{\top} B(\textbf{c}) Z\right), \\
    \end{aligned}
\end{equation*}
and
\begin{equation*}
    \operatorname{diag}\left\{\lambda_1 t_{11}, \lambda_2 t_{22}, \cdots, \lambda_n t_{n n}\right\} = \Lambda \odot\left(Q^{\top} B(\textbf{c}) Z\right),
\end{equation*}
where the diagonal matrix
\begin{equation}\label{diaL}
  \Lambda=\operatorname{diag}\left\{\lambda_{1}, \lambda_{2}, \cdots, \lambda_{ n}\right\},  
\end{equation}
and $\odot$ denotes the Hadamard product of the matrices.
Thus, Eq.(\ref{lts}) is equivalent to
\begin{equation}\label{3-1}
    \Lambda \odot\left(Q^{\top} B(\textbf{c}) Z\right) - I_n \odot\left(Q^{\top} A(\textbf{c}) Z\right)=\textbf{0}.
\end{equation}
Noting that $S$ and $T$ are upper triangular matrices, it holds
\begin{equation}\label{3-2}
    \| P \odot S\|_F = 0,\quad \text{and} \quad \| P \odot T\|_F = 0,
\end{equation}
where $P$ is an $n\times n$ mask matrix defined as 
\begin{equation}\label{P}
    P=\left(\begin{array}{cccccc}
        0 & 0 & 0 & \cdots & 0& 0 \\
        1& 0 & 0 & \cdots & 0& 0 \\
        1& 1 & 0& \cdots & 0& 0 \\
        \vdots & \vdots & \vdots & \ddots & \vdots & \vdots \\
        1 & 1& 1&\cdots & 0& 0\\
        1 & 1& 1&\cdots & 1& 0
        \end{array}\right).
\end{equation}
Combining Eq.(\ref{3-1}) and Eq.(\ref{3-2}), given the PGIEP optimization model on the product manifolds
\begin{equation}\label{model}
    \min_{(\textbf{c},Q,Z)\in \mathcal{M}} F(\textbf{c},Q,Z) = \frac{1}{2}\| \Lambda \odot\left(Q^{\top} B(\textbf{c}) Z\right) - I_n \odot\left(Q^{\top} A(\textbf{c}) Z\right) \|_F^2 + \frac{1}{2}\| P \odot (Q^{\top} A(\textbf{c}) Z)\|_F^2 + \frac{1}{2}\| P \odot (Q^{\top} B(\textbf{c}) Z)\|_F^2,
\end{equation}
where $\mathcal{M}=\mathbb{R}^n \times \mathcal{O}(n)\times \mathcal{O}(n)$.

In Eq.(\ref{st}), the condition $t_{ii} \neq 0$ is required, as implied by Theorem \ref{GRSD}, which is guaranteed by the fact that $B(\mathbf{c})$ is a nonsingular matrix.
We now consider the case where $B(\mathbf{c})$ is singular and the generalized eigenvalue problem $A(\mathbf{c})x = \lambda B(\mathbf{c})x$ has an infinite eigenvalue, denoted by $\lambda_n = \infty$ without loss of generality.

Given $n-1$ real and finite numbers $\lambda_1 \leqslant \lambda_2 \leqslant\cdots \leqslant\lambda_{n-1}$ and $\lambda_n=\infty$, it follows that $t_{nn} = 0$.
Then the spectrum of the generalized eigenvalue problem \( A(\textbf{c})x = \lambda B(\textbf{c})x \) is given by
\begin{equation}\label{stt}
    \sigma(A(\textbf{c}),B(\textbf{c}))=\left\{\frac{s_{11}}{t_{11}},\frac{s_{22}}{t_{22}}\dots,\frac{s_{n-1,n-1}}{t_{n-1,n-1}} \right\} \bigcup \left\{\infty \right\}.
\end{equation}
A direct result of Eq.(\ref{stt}) is
\begin{equation}\label{ltss}
    \lambda_kt_{kk}=s_{kk},\quad k=1,\dots,n-1.
\end{equation}
The diagonal matrix generated by the elements \( s_{kk} \) is denoted as
\begin{equation*}
    \begin{aligned} 
    \operatorname{diag}\left\{s_{11},s_{22},\cdots,s_{n-1,n-1}\right\}= \left[I_n \odot\left(Q^{\top} A(\textbf{c}) Z\right)\right]_{nn}, \\ 
    \end{aligned}
\end{equation*}
where $[*]_{nn}$ is the $(n-1) \times (n-1)$ submatrix of the $n \times n$ matrix $*$ obtained by removing the $n$-th row and $n$-th column.
Similarly, we get 
\begin{equation*}
    \begin{aligned} 
    \operatorname{diag}\left\{t_{11},t_{22}, \cdots,t_{n-1,n-1}\right\}=\left[I_n \odot\left(Q^{\top} B(\textbf{c}) Z\right)\right]_{nn}, \\
    \end{aligned}
\end{equation*}
and
\begin{equation*}
    \operatorname{diag}\left\{\lambda_1 t_{11}, \lambda_2 t_{22},\cdots,\lambda_{n-1} t_{n-1,n-1}\right\} = \left[\Lambda \odot\left(Q^{\top} B(\textbf{c}) Z\right)\right]_{nn}.
\end{equation*}
Thus, Eq.(\ref{ltss}) is equivalent to
\begin{equation}\label{3-ts}
    \left[\Lambda \odot\left(Q^{\top} B(\textbf{c}) Z\right)\right]_{nn} - \left[I_n \odot\left(Q^{\top} A(\textbf{c}) Z\right)\right]_{nn}=\textbf{0}.
\end{equation}
Consequently, when the matrix $B(\mathbf{c})$ is singular and $\lambda_n = \infty$, we obtain the PGIEP optimization model formulated on the product manifolds
\begin{equation}\label{model2}
\begin{aligned}
    \min_{(\textbf{c},Q,Z)\in \mathcal{M}} F(\textbf{c},Q,Z) &= \frac{1}{2}\| \left[\Lambda \odot\left(Q^{\top} B(\textbf{c}) Z\right)\right]_{nn} - \left[I_n \odot\left(Q^{\top} A(\textbf{c}) Z\right)\right]_{nn} \|_F^2 
     \\
     &\quad + \frac{1}{2}\| P \odot (Q^{\top} A(\textbf{c}) Z)\|_F^2 + \frac{1}{2}\| \hat{P} \odot (Q^{\top} B(\textbf{c}) Z)\|_F^2,
\end{aligned}
\end{equation}
where \(\hat{P} = P + \textbf{e}_n \textbf{e}_n^\top\) (with \(\mathbf{e}_n\) being the \(n\)-th column of identity matrix) enforces minimization of the \((n,n)\)-th entry of \(Q^\top B(\mathbf{c}) Z\) must be minimized, i.e., \(\left(Q^\top B(\mathbf{c}) Z\right)_{nn} = t_{nn}\).

Therefore, this naturally leads us to solve the PGIEP via optimization, computing $Q$ and $Z$ using orthogonality-preserving algorithms.
The strong scalability of neural networks facilitates the simultaneous optimization of all variables in Eq. (\ref{model}) and Eq. (\ref{model2}). 
Optimizers in neural networks are crucial, as they iteratively adjust trainable parameters to minimize the objective function through gradient-based methods. Among these, the Adam optimizer \cite{kingma2014adam} is highly effective and widely adopted.
A fundamental requirement for guaranteeing the convergence of the Adam algorithm is the Lipschitz continuity of the objective function’s gradients. 
Consequently, we proceed to discuss the Lipschitz continuity of the gradients in Eq. (\ref{model}) and Eq. (\ref{model2}).

\subsection{Gradient Lipschitz continuity.}
In this section, we establish the gradient Lipschitz continuity of the objective function in Eq.(\ref{model}). A similar result also holds for the objective function in Eq.(\ref{model2}). To this end, we need the following lemmas.

\begin{lemma}\label{le.1}
    (Heine-Borel theorem \cite{hayes1956heine,veblen1904heine})
    Let \( S \) be a subset of \( \mathbb{R}^n \). The following two assertions are equivalent:
    
(1) \( S \) is compact, meaning every open cover of \( S \) has a finite subcover.

(2) \( S \) is closed and bounded.
\end{lemma}
\begin{lemma}
    The set \( \bar{\Omega} = \{ (\textbf{c}, Q, Z) \mid Q \in \mathcal{O}(n), Z \in \mathcal{O}(n), \textbf{c} \in B_M \} \) is compact, where \( B_M = \{ \mathbf{v} \in \mathbb{R}^n \mid \|\mathbf{v}\|_2 \leqslant M \} \), and \( M \geqslant 0 \) is a constant.
\end{lemma}
\begin{proof}
We prove the compactness of \( \bar{\Omega} \) in two steps.
As the first step, we show that \( \bar{\Omega} \) is bounded.
For any \( (\textbf{c}, Q, Z) \in \bar{\Omega} \), it holds
\begin{equation*}
\|Q\|_F = \sqrt{n}, \quad \|Z\|_F = \sqrt{n}, \quad \text{and} \quad \|\textbf{c}\|_2 \leqslant M.
\end{equation*}
Thus, the norm of \( (\textbf{c}, Q, Z) \), defined as 
\begin{equation*}
\|(\textbf{c}, Q, Z)\| := \sqrt{\|Q\|_F^2 + \|Z\|_F^2 + \|\textbf{c}\|_2^2}
\end{equation*}
is bounded from above by $\sqrt{2n + M^2}$,
that is to say, \( \bar{\Omega} \) is bounded. 

As the second step, we prove that \( \bar{\Omega} \) is closed.
By definition, both $\mathcal{O}(n)$ and $B_M$ are closed. 
Therefore, $\bar{\Omega}$, as the Cartesian product of $\mathcal{O}(n)$ and $B_M$, is closed. By Lemma \ref{le.1}, \( \bar{\Omega} \) is compact.   

\end{proof}

\begin{lemma}\label{le.3}
    Let \( Y, W \in \mathbb{R}^{n \times n} (n\geqslant 2) \). The following inequality holds
    \[
    \| Y \odot W \|_F \leqslant \| Y \|_F \| W \|_F.
    \]
Specifically, $\| I_n \odot W \|_F \leqslant \| W \|_F$\textup{;} $\| \Lambda \odot W \|_F \leqslant \| \Lambda \|_\infty \| W \|_F$ for \( \Lambda \) defined in Eq.(\ref{diaL}) and $ \| P \odot W \|_F \leqslant \sqrt{n-1} \| W \|_F$ for \( P \) defined in Eq.(\ref{P}).
\end{lemma}
\begin{proof}
Using the Cauchy-Schwarz inequality, it follows 
\begin{equation}\label{le.3.1}
    \| Y \odot W \|_F = \sqrt{\sum_{i,j=1}^{n} (Y_{ij} W_{ij})^2} \leqslant \sqrt{\left( \sum_{i,j=1}^{n} Y_{ij}^2 \right) \left( \sum_{i,j=1}^{n} W_{ij}^2 \right)} = \| Y \|_F \| W \|_F.
\end{equation} 
Then the detailed results follow from the application of inequality (\ref{le.3.1}).

\end{proof}

\begin{lemma}\label{le.6}
    Let $R_0 = \{ R \in \mathbb{R}^{n \times n} \mid \|R\|_F \leqslant M_0 \}$, where \(M_0\) is a nonnegative real constant. Suppose that \(A_i, B_i, \Lambda \in R_0\) for \(i = 0,\dots, n\) and that for any \(\textbf{c} \in B_M\) we have \(A(\textbf{c}), B(\textbf{c}) \in R_0\).
  Then the gradient of the objective function \(F\) with respect to \(Q\),
    \[
    \nabla_Q F: B_M \times \mathcal{O}(n) \times \mathcal{O}(n) \rightarrow \mathbb{R}^{n \times n}
    \]
 is Lipschitz continuous, for fixed \((\textbf{c}, Z) \in B_M \times \mathcal{O}(n)\). 
    That is, for any \(Q, \hat{Q} \in \mathcal{O}(n)\), there exists a constant \(L_Q > 0\) such that 
    \[
    \|\nabla_Q F(\textbf{c}, Q, Z) - \nabla_Q F(\textbf{c}, \hat{Q}, Z)\|_F \leqslant L_Q \|Q - \hat{Q}\|_F.
    \]
\end{lemma}

\begin{proof}
In the optimization model \textup{(\ref{model})}, the gradient of the objective function \(F\) with respect to \(Q\) is given by
\begin{equation}\label{tiduQ}
        \begin{aligned}
        \nabla_Q F &= A(\textbf{c}) Z \Bigl(I_n \odot \Bigl((A(\textbf{c}) Z)^{\top} Q\Bigr) - \Lambda \odot \Bigl((B(\textbf{c}) Z)^{\top} Q\Bigr)\Bigr) 
        \\ & \quad - B(\textbf{c}) Z \Bigl(\Bigl(I_n \odot \Bigl((A(\textbf{c}) Z)^{\top} Q\Bigr) - \Lambda \odot \Bigl((B(\textbf{c}) Z)^{\top} Q\Bigr)\Bigr) \odot \Lambda\Bigr)\\
        &\quad + A(\textbf{c}) Z \Bigl(P^{\top} \odot \Bigl((A(\textbf{c}) Z)^{\top} Q\Bigr)\Bigr) \\
        & \quad + B(\textbf{c}) Z \Bigl(P^{\top} \odot \Bigl((B(\textbf{c}) Z)^{\top} Q\Bigr)\Bigr)\\
        &:= \nabla_Q F_1 + \nabla_Q F_2 + \nabla_Q F_3 + \nabla_Q F_4.
    \end{aligned}
\end{equation}
    
    We consider first the term $\nabla_Q F_1$.
    For any \(Q, \hat{Q} \in \mathcal{O}(n)\), by Lemma \ref{le.3} it holds
    \begin{equation}\label{4.4.1}
    \begin{aligned}
    &\quad \ \|\nabla_Q F_1(Q) - \nabla_Q F_1(\hat{Q})\|_F \\
    & = \Bigl\| A(\textbf{c}) Z \Bigl[ \Bigl( I_n \odot \bigl((A(\textbf{c}) Z)^{\top} Q\bigr) - \Lambda \odot \bigl((B(\textbf{c}) Z)^{\top} Q\bigr) \Bigr) - \Bigl( I_n \odot \bigl((A(\textbf{c}) Z)^{\top} \hat{Q}\bigr) - \Lambda \odot \bigl((B(\textbf{c}) Z)^{\top} \hat{Q}\bigr) \Bigr) \Bigr] \Bigr\|_F \\
    & \leqslant \|A(\textbf{c}) Z\|_F \Bigl\| \Bigl( I_n \odot \bigl((A(\textbf{c}) Z)^{\top} Q\bigr) - \Lambda \odot \bigl((B(\textbf{c}) Z)^{\top} Q\bigr) \Bigr)  - \Bigl( I_n \odot \bigl((A(\textbf{c}) Z)^{\top} \hat{Q}\bigr) - \Lambda \odot \bigl((B(\textbf{c}) Z)^{\top} \hat{Q}\bigr) \Bigr) \Bigr\|_F \\
    & \leqslant \|A(\textbf{c})\|_F \Bigl( \bigl\| I_n \odot \bigl((A(\textbf{c}) Z)^{\top} (Q-\hat{Q})\bigr) \bigr\|_F  + \bigl\| \Lambda \odot \bigl((B(\textbf{c}) Z)^{\top} (Q-\hat{Q})\bigr) \bigr\|_F \Bigr) \\
    & \leqslant \|A(\textbf{c})\|_F \Bigl( \|A(\textbf{c})^{\top}\|_F \|Q-\hat{Q}\|_F 
    + \|\Lambda\|_{\infty} \|B(\textbf{c})^{\top}\|_F \|Q-\hat{Q}\|_F \Bigr) \\
    & \leqslant (M_0^2 + M_0^2 \|\Lambda\|_{\infty}) \|Q-\hat{Q}\|_F.
    \end{aligned}
    \end{equation}

Concerning the term $\nabla_Q F_2$, a similar analysis also shows, for any \(Q, \hat{Q} \in \mathcal{O}(n)\),
    \begin{equation}\label{4.4.3}
    \begin{aligned}
   \|\nabla_Q F_2(Q) - \nabla_Q F_2(\hat{Q})\|_F 
\leqslant \bigl( M_0^2 \|\Lambda\|_{\infty} + M_0^2 \|\Lambda\|_{\infty}^2 \bigr) \|Q-\hat{Q}\|_F.
    \end{aligned}
    \end{equation}
    
    In the same manner, we get
     \begin{equation}\label{4.4.3}
    \|\nabla_Q F_3(Q) - \nabla_Q F_3(\hat{Q})\|_F\leqslant M_0^2 \sqrt{n-1} \|Q-\hat{Q}\|_F,
    \end{equation}
    and
    \begin{equation}\label{4.4.4}
    \|\nabla_Q F_4(Q) - \nabla_Q F_4(\hat{Q})\|_F\leqslant M_0^2 \sqrt{n-1} \|Q-\hat{Q}\|_F.
    \end{equation}

    Thus, for any \(Q, \hat{Q} \in \mathcal{O}(n)\), combining  Eq. \eqref{4.4.1} -- Eq.\eqref{4.4.4} yields
    \[
    \begin{aligned}
    &\quad \ \|\nabla_Q F(Q) - \nabla_Q F(\hat{Q})\|_F \\
        & \leqslant \sum_{i=1}^{4}\|\nabla_Q F_i(Q) - \nabla_Q F_i(\hat{Q})\|_F  \\
    & = \Bigl( M_0^2 + 2M_0^2\|\Lambda\|_{\infty} + M_0^2\|\Lambda\|_{\infty}^2 + 2 \sqrt{n-1}M_0^2 \Bigr) \|Q-\hat{Q}\|_F \\
    & := L_Q \|Q-\hat{Q}\|_F.
    \end{aligned}
    \]
    This completes the proof.
    \end{proof}
    
    \begin{lemma}\label{le.7}
         Suppose that \(A_i, B_i, \Lambda \in R_0\) for \(i = 0,\dots, n\) and that for any \(\textbf{c} \in B_M\) we have \(A(\textbf{c}), B(\textbf{c}) \in R_0\).
        Then the function
        \[
        \nabla_Z F: B_M \times \mathcal{O}(n) \times \mathcal{O}(n) \rightarrow \mathbb{R}^{n \times n},
        \]
        where $\nabla_Z F$ denotes the gradient of the objective function \(F\) in Eq.(\ref{model}) with respect to \(Q\), is Lipschitz continuous with respect to \(Z \in \mathcal{O}(n)\) for fixed \((\textbf{c}, Q) \in B_M \times \mathcal{O}(n)\). That is, for any \(Z, \hat{Z} \in \mathcal{O}(n)\), there exists a constant \(L_Z > 0\) such that 
        \[
        \|\nabla_Z F(\textbf{c}, Q, Z) - \nabla_Z F(\textbf{c}, Q, \hat{Z})\|_F \leqslant L_Z \|Z - \hat{Z}\|_F.
        \]
    \end{lemma}
\begin{proof}
The proof follows a similar approach to that of Lemma \ref{le.6}.   
\end{proof}

\begin{lemma}\label{le.8}
    Suppose that \(A_i, B_i, \Lambda \in R_0\) for \(i = 0,\dots, n\) and that for any \(\textbf{c} \in B_M\) it holds \(A(\textbf{c}), B(\textbf{c}) \in R_0\).
    Then, the gradient of the objective function \(F\) in Eq.(\ref{model}) with respect to the vector \(\textbf{c}\),
    \[
    \nabla_{\textbf{c}} F: B_M \times \mathcal{O}(n) \times \mathcal{O}(n) \rightarrow \mathbb{R}^{n}
    \]
  is Lipschitz continuous with respect to \(\textbf{c} \in B_M\) for fixed \((Z, Q) \in \mathcal{O}(n) \times \mathcal{O}(n)\). That is, for any \(\textbf{c}, \hat{\textbf{c}} \in B_M\), there exists a constant \(L_{\textbf{c}} > 0\) such that 
    \[
    \|\nabla_{\textbf{c}} F(\textbf{c}, Q, Z) - \nabla_{\textbf{c}} F(\hat{\textbf{c}}, Q, Z)\|_2 \leqslant L_{\textbf{c}} \|\textbf{c} - \hat{\textbf{c}}\|_2.
    \]
\end{lemma}
\begin{proof}
We first compute the gradient of the objective function \(F\) with respect to the vector \(\textbf{c}\).
Define \( M_1(\textbf{c}) = I_n \odot (Q^\top A(\textbf{c}) Z) - \Lambda \odot (Q^\top B(\textbf{c}) Z) \).
Then a direct computation gives
    \[
    \frac{\partial A(\textbf{c})}{\partial c_k} = A_k, \quad \frac{\partial B(\textbf{c})}{\partial c_k} = B_k,\quad \text{and} \quad \frac{\partial M_{1_{ij}}}{\partial c_k} = \Bigl( I_n \odot (Q^{\top} A_k Z) - \Lambda \odot (Q^{\top} B_k Z) \Bigr)_{ij},\ k=1,\dots,n.
    \]
    Calculating the derivative of 
    \[
    f_1(\textbf{c}) := \frac{1}{2} \Bigl\| I_n \odot \bigl(Q^{\top} A(\textbf{c}) Z\bigr) - \Lambda \odot \bigl(Q^{\top} B(\textbf{c}) Z\bigr) \Bigr\|_F^2
    \]
    with respect to \(c_k\), one has
    \[
    \begin{aligned}
    \frac{\partial f_1(\textbf{c})}{\partial c_k} 
    &= \sum_{i,j} M_{1_{ij}} \frac{\partial M_{1_{ij}}}{\partial c_k} \\
    &= \sum_{i,j} M_{1_{ij}} \Bigl( I_n \odot (Q^{\top} A_k Z) - \Lambda \odot (Q^{\top} B_k Z) \Bigr)_{ij} \\
    &= \Bigl\langle I_n \odot (Q^{\top} A(\textbf{c}) Z) - \Lambda \odot (Q^{\top} B(\textbf{c}) Z), \; I_n \odot (Q^{\top} A_k Z) - \Lambda \odot (Q^{\top} B_k Z) \Bigr\rangle,
    \end{aligned}
    \]
    with \(\langle \cdot, \cdot \rangle\) denoting the inner product of matrices.
    Denoting by \( M_2(\textbf{c}) = P \odot (Q^{\top} A(\textbf{c}) Z) \), one then has
    \[
    \frac{\partial M_{2_{ij}}}{\partial c_k} = \Bigl( P \odot (Q^{\top} A_k Z) \Bigr)_{ij}.
    \]
    The derivative of 
    \[
    f_2(\textbf{c}) := \frac{1}{2} \Bigl\| P \odot \bigl(Q^{\top} A(\textbf{c}) Z\bigr) \Bigr\|_F^2
    \]
    in \(c_k\) is then obtained as
    \[
    \begin{aligned}
    \frac{\partial f_2(\textbf{c})}{\partial c_k} 
    &= \sum_{i,j} M_{2_{ij}} \frac{\partial M_{2_{ij}}}{\partial c_k} \\
    &= \sum_{i,j} \Bigl( P \odot (Q^{\top} A(\textbf{c}) Z) \Bigr)_{ij} \Bigl( P \odot (Q^{\top} A_k Z) \Bigr)_{ij} \\
    &= \Bigl\langle P \odot (Q^{\top} A(\textbf{c}) Z), \; P \odot (Q^{\top} A_k Z) \Bigr\rangle.
    \end{aligned}
    \]
    Similarly, denoting by \( M_3(\textbf{c}) = P \odot (Q^{\top} B(\textbf{c}) Z) \),
    for $f_3(\textbf{c}) := \frac{1}{2} \Bigl\| P \odot \bigl(Q^{\top} B(\textbf{c}) Z\bigr) \Bigr\|_F^2$, it holds
    \[
    \begin{aligned}
    \frac{\partial f_3(\textbf{c})}{\partial c_k} 
        &= \sum_{i,j} M_{3_{ij}} \frac{\partial M_{3_{ij}}}{\partial c_k} \\
    &= \sum_{i,j} \Bigl( P \odot (Q^{\top} B(\textbf{c}) Z) \Bigr)_{ij} \Bigl( P \odot (Q^{\top} B_k Z) \Bigr)_{ij} \\
    &= \Bigl\langle P \odot (Q^{\top} B(\textbf{c}) Z), \; P \odot (Q^{\top} B_k Z) \Bigr\rangle.
    \end{aligned}
    \]
    Combining the above results, the gradient vector is given by
    \[
    \begin{aligned}
    \nabla_{\textbf{c}} F &= \Bigl( \frac{\partial f_1(\textbf{c})}{\partial c_1} + \frac{\partial f_2(\textbf{c})}{\partial c_1} + \frac{\partial f_3(\textbf{c})}{\partial c_1},\; \cdots,\; \frac{\partial f_1(\textbf{c})}{\partial c_n} + \frac{\partial f_2(\textbf{c})}{\partial c_n} + \frac{\partial f_3(\textbf{c})}{\partial c_n} \Bigr)^{\top} := \Bigl( \frac{\partial f}{\partial c_1}, \cdots, \frac{\partial f}{\partial c_n} \Bigr)^{\top}.
    \end{aligned}
    \]
    For any \(\textbf{c}, \hat{\textbf{c}} \in B_M\), we have
    \begin{equation}\label{cc}
    \nabla_{\textbf{c}} F(\textbf{c},Q,Z) - \nabla_{\textbf{c}} F(\hat{\textbf{c}},Q,Z) 
    = \left(\frac{\partial f}{\partial c_1} (\textbf{c}) - \frac{\partial f}{\partial c_1} (\hat{\textbf{c}}),\, \cdots,\, \frac{\partial f}{\partial c_n} (\textbf{c}) - \frac{\partial f}{\partial c_n} (\hat{\textbf{c}})\right)^{\top}.
    \end{equation}
    The \(k\)th component of Eq. (\ref{cc}) is given by
    \begin{equation}\label{fc}
        \begin{aligned}
    \frac{\partial f}{\partial c_k} (\textbf{c}) - \frac{\partial f}{\partial c_k} (\hat{\textbf{c}})
    & = \Bigl\langle M_1(\textbf{c}) - M_1(\hat{\textbf{c}}),\, I_n \odot \bigl(Q^{\top} A_k Z\bigr) - \Lambda \odot \bigl(Q^{\top} B_k Z\bigr) \Bigr\rangle  + \Bigl\langle M_2(\textbf{c}) - M_2(\hat{\textbf{c}}),\, P \odot \bigl(Q^{\top} A_k Z\bigr) \Bigr\rangle \\
    & \quad + \Bigl\langle M_3(\textbf{c}) - M_3(\hat{\textbf{c}}),\, P \odot \bigl(Q^{\top} B_k Z\bigr) \Bigr\rangle,
        \end{aligned}
    \end{equation}
    where 
    \[
    \begin{aligned}
    &M_1(\textbf{c}) - M_1(\hat{\textbf{c}}) = I_n \odot \bigl(Q^{\top} (\Delta A) Z\bigr) - \Lambda \odot \bigl(Q^{\top} (\Delta B) Z\bigr):=\Delta M_1, \\
    &M_2(\textbf{c}) - M_2(\hat{\textbf{c}}) = P \odot \bigl(Q^{\top} (A(\textbf{c}) - A(\hat{\textbf{c}})) Z\bigr)
    = P \odot \bigl(Q^{\top} (\Delta A) Z\bigr):=\Delta M_2, \\
   &M_3(\textbf{c}) - M_3(\hat{\textbf{c}}) = P \odot \bigl(Q^{\top} (B(\textbf{c}) - B(\hat{\textbf{c}})) Z\bigr)
    = P \odot \bigl(Q^{\top} (\Delta B) Z\bigr):=\Delta M_3,
    \end{aligned}
    \]
    with
    \[
    \Delta A = A(\textbf{c}) - A(\hat{\textbf{c}}) \quad \text{and} \quad \Delta B = B(\textbf{c}) - B(\hat{\textbf{c}}).
    \]
    Moreover, by the Cauchy-Schwartz inequality we have
    \begin{equation}\label{M123}
    \begin{aligned}
    \Bigl|\Bigl\langle \Delta M_1,\, I_n \odot (Q^{\top} A_k Z) - \Lambda \odot (Q^{\top} B_k Z) \Bigr\rangle\Bigr|
    &\leqslant \|\Delta M_1\|_F \, \|I_n \odot (Q^{\top} A_k Z) - \Lambda \odot (Q^{\top} B_k Z)\|_F, \\
    \Bigl|\Bigl\langle \Delta M_2,\, P \odot (Q^{\top} A_k Z) \Bigr\rangle\Bigr|
    &\leqslant \|\Delta M_2\|_F \, \|P \odot (Q^{\top} A_k Z)\|_F, \\
    \Bigl|\Bigl\langle \Delta M_3,\, P \odot (Q^{\top} B_k Z) \Bigr\rangle\Bigr|
    &\leqslant \|\Delta M_3\|_F \, \|P \odot (Q^{\top} B_k Z)\|_F.
    \end{aligned}
    \end{equation}
    Thus, by Eq. (\ref{fc}) and Eq. (\ref{M123}) we obtain
    \begin{equation}\label{th4.2.1}
        \begin{aligned}
    \left|\frac{\partial f}{\partial c_k} (\textbf{c}) - \frac{\partial f}{\partial c_k} (\hat{\textbf{c}})\right|
    &\leqslant \|\Delta M_1\|_F\, \|I_n \odot (Q^{\top} A_k Z) - \Lambda \odot (Q^{\top} B_k Z)\|_F  + \|\Delta M_2\|_F\, \|P \odot (Q^{\top} A_k Z)\|_F \\
    &\quad + \|\Delta M_3\|_F\, \|P \odot (Q^{\top} B_k Z)\|_F.
        \end{aligned}
    \end{equation}
    Since
    \[
    \|\Delta A\|_F \leqslant \sqrt{n}M_0 \|\textbf{c} - \hat{\textbf{c}}\|_2 \quad \text{and} \quad \|\Delta B\|_F \leqslant \sqrt{n}M_0 \|\textbf{c} - \hat{\textbf{c}}\|_2,
    \]
    and by Lemma \ref{le.3} it follows
    \begin{equation}\label{th4.2.2}
        \begin{aligned}
    \|\Delta M_1\|_F &\leqslant \|\Delta A\|_F + \|\Lambda\|_{\infty} \|\Delta B\|_F
    \leqslant \left(\sqrt{n}M_0 + \sqrt{n}M_0\|\Lambda\|_{\infty}\right) \|\textbf{c} - \hat{\textbf{c}}\|_2, \\
    \|\Delta M_2\|_F &= \|P \odot (Q^{\top} \Delta A\, Z)\|_F \leqslant \sqrt{n}M_0 \sqrt{n-1} \|\textbf{c} - \hat{\textbf{c}}\|_2, \\
    \|\Delta M_3\|_F &= \|P \odot (Q^{\top} \Delta B\, Z)\|_F \leqslant \sqrt{n}M_0 \sqrt{n-1} \|\textbf{c} - \hat{\textbf{c}}\|_2.
        \end{aligned}
    \end{equation}
Denote by \(c_k\) the component of $\mathbf{c}$ such that
\[
\left| \frac{\partial f}{\partial c_k}(\mathbf{c}) - \frac{\partial f}{\partial c_k}(\hat{\mathbf{c}}) \right|=\max_{1 \leqslant i \leqslant n} \left| \frac{\partial f}{\partial c_i}(\mathbf{c}) - \frac{\partial f}{\partial c_i}(\hat{\mathbf{c}}) \right|.
\]  
    Then, by \eqref{th4.2.1} and \eqref{th4.2.2}, for any \(\textbf{c}, \hat{\textbf{c}} \in B_M\), we have
    \[
    \begin{aligned}
    &\quad \ \|\nabla_{\textbf{c}} F(\textbf{c}) - \nabla_{\textbf{c}} F(\hat{\textbf{c}})\|_2\\
    &\leqslant \sqrt{n} \|\Delta M_1\|_F \|I_n \odot (Q^{\top} A_k Z) - \Lambda \odot (Q^{\top} B_k Z)\|_F  + \sqrt{n}\|\Delta M_2\|_F \|P \odot (Q^{\top} A_k Z)\|_F  \\
    &\quad \ + \sqrt{n}\|\Delta M_3\|_F \|P \odot (Q^{\top} B_k Z)\|_F \\
    &\leqslant \sqrt{n} \Bigl( \|\Delta M_1\|_F \bigl(\|A_k\|_F + \|\Lambda\|_{\infty} \|B_k\|_F\bigr) + \|\Delta M_2\|_F \sqrt{n-1} \|A_k\|_F  + \|\Delta M_3\|_F \sqrt{n-1} \|B_k\|_F \Bigr) \\
    &\leqslant \sqrt{n} \Bigl( \bigl(\sqrt{n}M_0 + \sqrt{n}M_0 \|\Lambda\|_{\infty}\bigr)(M_0 + M_0 \|\Lambda\|_{\infty}) + \sqrt{n}M_0 \sqrt{n-1} M_0 + \sqrt{n}M_0\sqrt{n-1} M_0 \Bigr) \|\textbf{c} - \hat{\textbf{c}}\|_2 \\
    &= n\Bigl((M_0 + M_0 \|\Lambda\|_{\infty})^2 + 2M_0^2 \sqrt{n-1}\Bigr) \|\textbf{c} - \hat{\textbf{c}}\|_2 \\
    &:= L_{\textbf{c}} \|\textbf{c} - \hat{\textbf{c}}\|_2.
    \end{aligned}
    \]
    This completes the proof.
    \end{proof}
    \begin{theorem}\label{th1}
        If \(A_i, B_i, \Lambda \in R_0\) for \(i = 0,1,\dots, n\) and for any \(\textbf{c} \in B_M\) it holds \(A(\textbf{c}), B(\textbf{c}) \in R_0\), then the objective function \(F\) in Eq.(\ref{model}) defined on \(\bar{\Omega}\) is gradient Lipschitz continuous. That is, there exists a constant \(L > 0\) such that for any \((\textbf{c}, Q, Z), (\hat{\textbf{c}}, \hat{Q}, \hat{Z}) \in \bar{\Omega}\) it holds that
        \begin{equation*}
        \|\nabla F(\textbf{c}, Q, Z) - \nabla F(\hat{\textbf{c}}, \hat{Q}, \hat{Z})\|_F \leqslant L\Bigl( \|Q-\hat{Q}\|_F + \|Z-\hat{Z}\|_F + \|\textbf{c}-\hat{\textbf{c}}\|_2 \Bigr).
        \end{equation*}
    \end{theorem}
    \begin{proof}
        For any \((\textbf{c}, Q, Z), (\hat{\textbf{c}}, \hat{Q}, \hat{Z}) \in \bar{\Omega}\), we have
        \[
        \begin{aligned}
        \|\nabla F(\textbf{c}, Q, Z) - \nabla F(\hat{\textbf{c}}, \hat{Q}, \hat{Z})\|_F 
        &\leqslant \|\nabla F(\textbf{c}, Q, Z) - \nabla F(\hat{\textbf{c}}, Q, Z)\|_F  + \|\nabla F(\hat{\textbf{c}}, Q, Z) - \nabla F(\hat{\textbf{c}}, \hat{Q}, Z)\|_F \\
        &\quad + \|\nabla F(\hat{\textbf{c}}, \hat{Q}, Z) - \nabla F(\hat{\textbf{c}}, \hat{Q}, \hat{Z})\|_F.
        \end{aligned}
        \]
        That is,
        \[
        \begin{aligned}
        \|\nabla F(\textbf{c}, Q, Z) - \nabla F(\hat{\textbf{c}}, \hat{Q}, \hat{Z})\|_F 
        &\leqslant \|\nabla_{\textbf{c}} F(\textbf{c}, Q, Z) - \nabla_{\textbf{c}} F(\hat{\textbf{c}}, Q, Z)\|_F  + \|\nabla_Q F(\hat{\textbf{c}}, Q, Z) - \nabla_Q F(\hat{\textbf{c}}, \hat{Q}, Z)\|_F \\
        &\quad + \|\nabla_Z F(\hat{\textbf{c}}, \hat{Q}, Z) - \nabla_Z F(\hat{\textbf{c}}, \hat{Q}, \hat{Z})\|_F.
        \end{aligned}
        \]
        By Lemmas \ref{le.6}, \ref{le.7} and \ref{le.8}, we obtain
        \[
        \|\nabla F(\textbf{c}, Q, Z) - \nabla F(\hat{\textbf{c}}, \hat{Q}, \hat{Z})\|_F \leqslant L_{\textbf{c}} \|\textbf{c} - \hat{\textbf{c}}\|_2 + L_Q \|Q - \hat{Q}\|_F + L_Z \|Z - \hat{Z}\|_F.
        \]
        Defining
        \[
        L = \max\{L_{\textbf{c}}, L_Q, L_Z\},
        \]
        it follows that
        \[
        \|\nabla F(\textbf{c}, Q, Z) - \nabla F(\hat{\textbf{c}}, \hat{Q}, \hat{Z})\|_F \leqslant L \Bigl( \|\textbf{c} - \hat{\textbf{c}}\|_2 + \|Q - \hat{Q}\|_F + \|Z - \hat{Z}\|_F \Bigr).
        \]
        This completes the proof.
        \end{proof}
               
Theorem \ref{th1} establishes the Lipschitz continuity of \(\nabla F\) over \(\bar{\Omega}\). 
Specifically, the Lipschitz continuity ensures stable variations of the objective function’s gradient.
Based on this theoretical result, model (\ref{model}) can be effectively solved by a variety of gradient-based algorithms.
Meanwhile, Theorem \ref{th1} also serves as a prerequisite assumption for the convergence of Adam algorithm\cite{chen2018convergence,gadat2022asymptotic,kavis2022high}.
The result of Theorem \ref{th1} holds as well for the objective function in Eq.(\ref{model2}) , and the proof is therefore omitted.
\begin{figure}[h]
    \centering
    \begin{subfigure}[b]{1.0\textwidth}
           \centering
           \includegraphics[width=\textwidth]{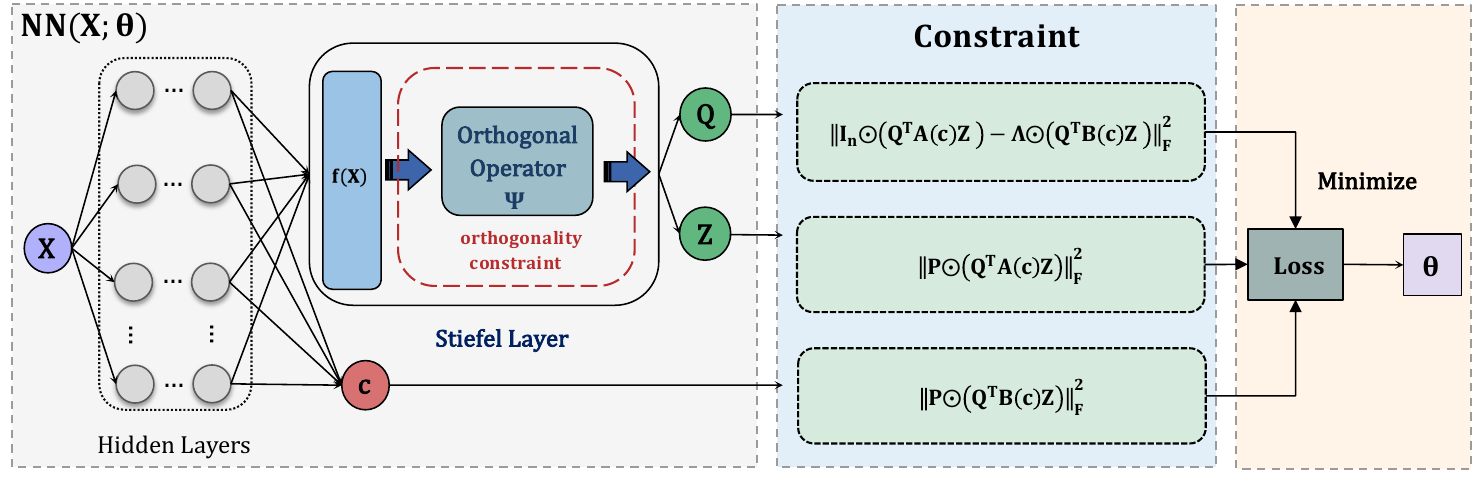}
       \end{subfigure}
    \caption{Schematic diagram of the P-SMLP architecture. In the Stiefel Layer, the output features from the previous layer are first linearly combined to generate new feature representations, then processed by the orthogonal operator $\Psi$ to strictly satisfy the orthogonality constraints on the Stiefel manifold, producing the output $Q$ and $Z$.}
    \label{PSMLP}
  \end{figure}
\subsection{Orthogonality-constrained neural network method}
Neural networks, with their powerful function approximation capabilities, can effectively handle complex nonlinear optimization problems. Meanwhile,leveraging the scalable network architecture and nonlinear activation functions, they can efficiently learn and model complex functional relationships, thereby facilitating end-to-end training. In this section, we design the P-SMLP method for solving PGIEPs.

In the context of PGIEP, the model (\ref{model}) is an optimization problem defined on the product manifold \(\mathcal{M} = \mathbb{R}^n \times \mathcal{O}(n) \times \mathcal{O}(n)\), where the optimization variables are the vector \(\textbf{c}\) and the orthogonal matrices \(Q\) and \(Z\). For the optimization of \(Q\) and \(Z\), we embed the orthogonality constraints directly into the neural network as a hard constraint. By employing the orthogonal decomposition operator \(\Psi\), the matrices produced by the network strictly satisfy the orthogonality requirement, ensuring that \(Q\) and \(Z\) remain on the Stiefel manifold throughout the optimization process.
Leveraging the flexibility of the neural network architecture, we can simultaneously output the parameterized vector \(\textbf{c}\).
Accordingly, we propose the P-SMLP neural network equipped with a Stiefel layer, as illustrated in Figure \ref{PSMLP}, which takes an orthogonal matrix \(X\) as input and produces orthogonal matrices \(Q\) and \(Z\), along with the vector \(\textbf{c}\). We denote the P-SMLP neural network by \(\mathcal{F} : \mathbb{R}^{n\times n} \to \bigl(\mathbb{R}^{n},\mathbb{R}^{n\times n}, \mathbb{R}^{n\times n}\bigr)\), which is formulated as follows
\begin{equation*}
\begin{aligned}
  (f(X), \textbf{c}) &:=  W_{\mathcal{L}} \circ \Phi \circ W_{\mathcal{L}-1} \circ \Phi \circ \cdots \circ \Phi \circ W_1 \circ \Phi \circ W_0(X), \\
  (Q, Z) &:= \Psi \circ f(X),
\end{aligned}
\end{equation*}
where \(\Psi\) denotes the orthogonal decomposition operator, \(W_l(x) = \omega_l x + b_l\) is an affine transformation with \(\omega_l\) as the weight matrix and \(b_l\) as the bias term (\(l = 0, 1, \ldots, \mathcal{L}\)), \(\Phi\) is the activation function, and \(\mathcal{L}\) represents the number of hidden layers.
  
The Stiefel layer in this network enforces the orthogonality of the matrices \(Q\) and \(Z\). This is achieved through the orthogonal decomposition operator \(\Psi\), which ensures that the matrices produced by the network adhere strictly to the Stiefel manifold constraint during the forward pass.
We exploit matrix decomposition operations within the PyTorch framework to acquire the orthogonal matrices \(Q\) and \(Z\), which combined with the network's output \(\textbf{c}\), yields the loss function
\begin{equation}\label{Losstol}
    \begin{aligned}
    \text{Loss} &\ \ = \gamma_1\| \Lambda \odot\left(Q^{\top} B(\textbf{c}) Z\right) - I_n \odot\left(Q^{\top} A(\textbf{c}) Z\right) \|_F^2 + \gamma_2\| P \odot (Q^{\top} A(\textbf{c}) Z)\|_F^2 + \gamma_3\| P \odot (Q^{\top} B(\textbf{c}) Z)\|_F^2\\
    &:= \text{Loss}_1+ \text{Loss}_2
\end{aligned}
\end{equation}
where $\text{Loss}_1=\gamma_1\| \Lambda \odot\left(Q^{\top} B(\textbf{c}) Z\right) - I_n \odot\left(Q^{\top} A(\textbf{c}) Z\right) \|_F^2$, $\text{Loss}_2=\gamma_2\| P \odot (Q^{\top} A(\textbf{c}) Z)\|_F^2 + \gamma_3\| P \odot (Q^{\top} B(\textbf{c}) Z)\|_F^2$, \(\gamma_1\), \(\gamma_2\) and \(\gamma_3\) are employed to balance the weights between the three loss terms. 
For the PGIEP with matrix $B(\mathbf{c})$ being singular, the loss function is defined as follows 
\begin{equation}\label{Losstol2}
    \begin{aligned}
    \text{Loss} & = \gamma_1\| \left[\Lambda \odot\left(Q^{\top} B(\textbf{c}) Z\right)\right]_{nn} - \left[I_n \odot\left(Q^{\top} A(\textbf{c}) Z\right)\right]_{nn} \|_F^2 + \gamma_2\| P \odot (Q^{\top} A(\textbf{c}) Z)\|_F^2 + \gamma_3\| \hat{P} \odot (Q^{\top} B(\textbf{c}) Z)\|_F^2.
\end{aligned}
\end{equation}
During backpropagation, automatic differentiation (AD, \cite{baydin2018automatic}) is used to compute the gradients. The Adam optimizer \cite{kingma2014adam} is then applied to minimize the loss function and obtain \(\textbf{c}\).

\begin{figure}[h]
    \centering
    \begin{subfigure}[b]{0.3\textwidth}  
      \centering
      \includegraphics[width=\textwidth]{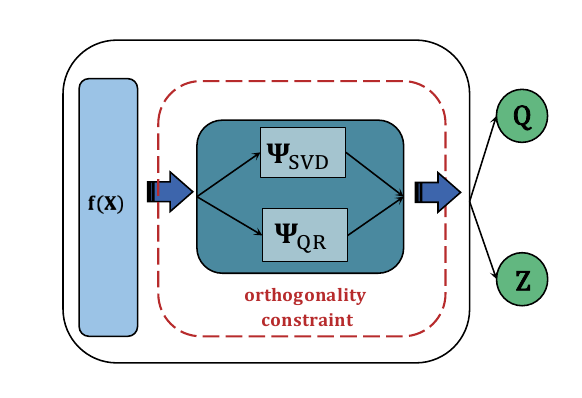}
      \subcaption{Strategy 1}
    \end{subfigure}%
    \hspace{0.02\textwidth}  
    \begin{subfigure}[b]{0.3\textwidth}  
      \centering
      \includegraphics[width=\textwidth]{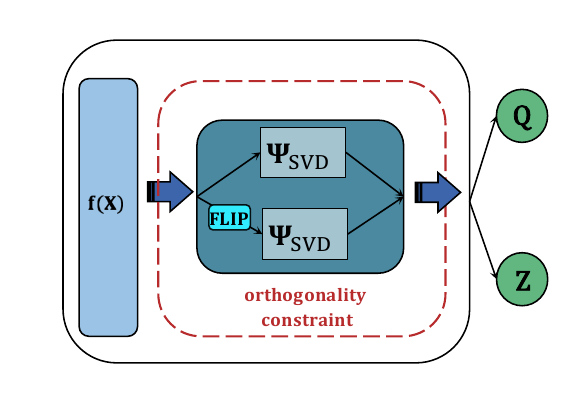}
      \subcaption{Strategy 2}
    \end{subfigure}%
    \hspace{0.02\textwidth} 
    \begin{subfigure}[b]{0.3\textwidth}  
      \centering
      \includegraphics[width=\textwidth]{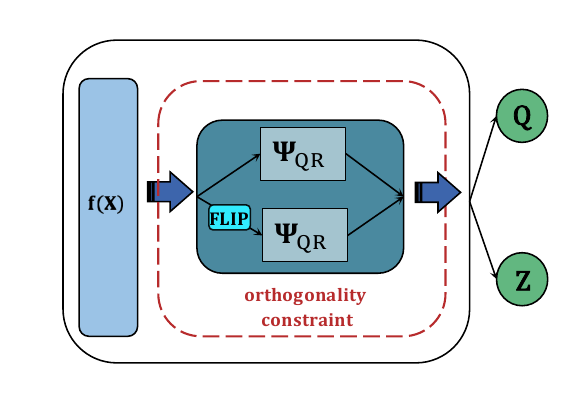}
      \subcaption{Strategy 3}
    \end{subfigure}
    \caption{Illustration of three strategies for generating orthogonal matrices using Stiefel layers, \text{FLIP:} \( f(X) \rightarrow f(X)T(k,k'). \)}
    \label{SVD+QR+FLIP}
\end{figure}

In general, within model (\ref{model}), the matrices \(Q\) and \(Z\) should be distinct. To address this, we propose three strategies for constructing the Stiefel layers.
Strategy 1 involves using both SVD and QR decomposition techniques to generate the orthogonal matrices \(Q\) and \(Z\) simultaneously.
Strategy 2 and Strategy 3 each employ a single orthogonal decomposition method (either SVD or QR). Additionally, the \( k \)-th and \( k' \)-th columns of \( f(X) \) are swapped, denoted as FLIP operation, i.e., \( f(X) T(k,k') \), where the matrix \( T(k,k') \) is obtained by swapping the \( k \)-th and \( k' \)-th rows of the identity matrix. This operation ensures the generation of distinct orthogonal matrices \( Q \) and \( Z \).
These three strategies are illustrated in Figure \ref{SVD+QR+FLIP}, and the P-SMLP algorithm for solving PGIEPs is provided in Algorithm \ref{alg}.

\vspace{-10pt}
\begin{figure}[h]
    \begin{center}
    \begin{minipage}{\textwidth}
    \begin{algorithm}[H]
    \caption{P-SMLP Optimization Algorithm for PGIEPs.}\label{alg}
        \begin{algorithmic}[1]
        \REQUIRE Initial orthogonal matrix \( X \)
        \ENSURE Parameter vector \( \textbf{c} := \{c_1, \dots, c_n\} \)
        \STATE Maximum number of training epochs \( N_{\text{epoch}} \), tolerance \( \varepsilon \), projection matrix \( P \), diagonal matrix \( \Lambda = \operatorname{diag}(\tilde{\sigma}) \), matrices \( A_j, B_j (0 \leqslant j \leqslant n) \), activation function \( \Phi \), number of hidden layers \( \mathcal{L} \), constants \( k, k' \), permutation matrix \( T(k,k') \), and orthogonal decomposition operator \( \Psi \).
        \STATE Randomly initialize the SMLP neural network parameters \( \theta(\boldsymbol{\omega}, \textbf{b}) := \{(\omega_l, b_l) \ | \ l = 0, \cdots, \mathcal{L} \} \) and set \( H_0 := X \).
        \FOR{Epoch = 1 \TO \( N_{\text{epoch}} \)}
            \FOR{\( l = 1 \) \TO \( \mathcal{L} \)}
                \STATE \( W_{l-1} = w_{l-1} \cdot H_{l-1} + b_{l-1}; \)
                \STATE \( H_l = \Phi \circ W_{l-1}; \)
            \ENDFOR
        \STATE \( (f(X), \textbf{c}) := W_{\mathcal{L}} \circ \Phi \circ W_{\mathcal{L}-1} \circ \Phi \circ \cdots \circ \Phi \circ W_1 \circ \Phi \circ W_0(X) \)
        \STATE \textbf{Stiefel layer:} \( (Q, Z) := \Psi \circ f(X) \), specifically: \begin{equation*}
        \begin{aligned}
            \text{Strategy 1:} & \quad Q \xleftarrow{} \Psi_{\text{SVD}} \circ f(X); \quad Z \xleftarrow{} \Psi_{\text{QR}} \circ f(X) \\
            \text{or}\quad \quad \quad \  \  \\
            \text{Strategy 2:} & \quad Q \xleftarrow{} \Psi_{\text{SVD}} \circ f(X); \quad Z \xleftarrow{} \Psi_{\text{SVD}} \circ ( f(X) T(k,k') ) \\
            \text{or}\quad \quad \quad \  \  \\
            \text{Strategy 3:} & \quad Q \xleftarrow{} \Psi_{\text{QR}} \circ f(X); \quad Z \xleftarrow{} \Psi_{\text{QR}} \circ ( f(X) T(k,k') )
            \end{aligned}
        \end{equation*}
        \STATE Update network parameters by minimizing the \( \text{Loss} \) in (\ref{Losstol}) or (\ref{Losstol2}), using AD in the PyTorch framework to compute gradients and update weights.
        \IF {\( \text{Loss} < \varepsilon \)}
            \STATE Break;
        \ENDIF
        \ENDFOR
        \RETURN \( \textbf{c} \)
        \end{algorithmic}
    \end{algorithm}
    \end{minipage}
    \end{center}
\end{figure}
\vspace{-20pt}

\section{Numerical examples }\label{4}
In this section, we test several examples of PGIEPs, including symmetric, asymmetric, and multiple- eigenvalue cases, as well as cases with $B(\textbf{c})$ being singular, to demonstrate the effectiveness of the proposed model and the P-SMLP method.
All experiments are conducted using the PyCharm IDE and the PyTorch 2.2.2 framework.
The neural network architecture is denoted as $[l_1,\dots,l_m]$, representing a neural network with $m$ hidden layers, where the $m$-th layer contains $l_m$ neurons.
The optimization of loss function is solved by the Adam method. 
We choose ReLU as the activation function and the training is stopped once the maximum epochs are attained. 
The learning rate is set to 0.001 if we do not specify otherwise.
Define \( k=1, k'=2 \) and \(\gamma_1=\gamma_2=\gamma_3=0.5\).
We use the $L^{\infty}$ error to measure the accuracy of the present method, which is defined as
\begin{equation*}
    \|\tilde{\sigma} - \sigma\|_{\infty} =  \max_{1\leqslant i \leqslant n} |\tilde{\lambda}_i-\lambda_i|, 
\end{equation*}
where $\tilde{\lambda}_i$ is the given eigenvalue and $\lambda_i \in \sigma$, with $\sigma$ being the set of generalized eigenvalues of \( (A(\textbf{c}), B(\textbf{c})) \), which are derived from the output \( \textbf{c} \) of the P-SMLP.

\begin{table}[h]
  \centering
  \caption{Example 4.1: Numerical results for the parameter vector \( \textbf{c} \) and  \( \|\tilde{\sigma} - \sigma\|_{\infty} \) and $\text{eps}$ represents the machine precision.}
  \begin{tabular}{>{\centering\arraybackslash}p{2.5cm} 
                  >{\centering\arraybackslash}p{3.0cm} 
                  >{\centering\arraybackslash}p{2.75cm} 
                  >{\centering\arraybackslash}p{2.75cm} 
                  >{\centering\arraybackslash}p{2.5cm}}  
    \toprule
    Strategy & \textbf{$\sigma=\{\lambda_1,\lambda_2\}$} & \textbf{$c_1$} & \textbf{$c_2$} & \textbf{$\|\tilde{\sigma}-\sigma\|_{\infty}$} \\ 
    \midrule
     1 & -1.0000, 3.0000 & -0.453003 & 0.361171 & $\text{eps}$ \\ 
     2 & -1.0000, 3.0000 & -0.453003 & 0.361171 & $\text{eps}$ \\ 
     3 & -1.0000, 3.0000 & -0.453003 & 0.361171 & $\text{eps}$ \\ 
    \bottomrule
  \end{tabular}
  \label{Example1}
\end{table}

\begin{figure}[h]
    \centering
    \begin{minipage}{0.32\textwidth}
      \centering
      \includegraphics[width=\linewidth]{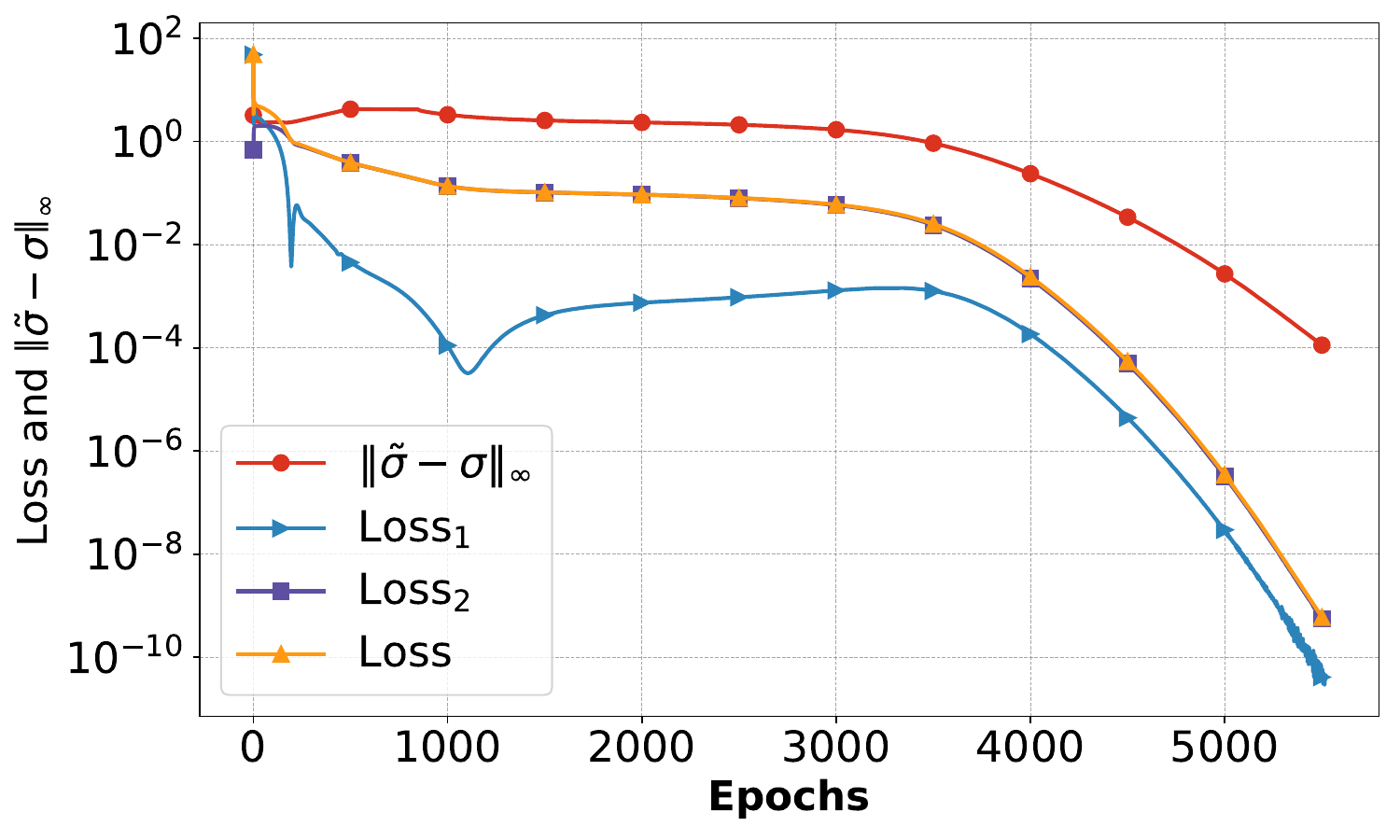}
      \subcaption{Strategy 1: Loss and $\|\tilde{\sigma}-\sigma\|_{\infty}$.}  
    \end{minipage}
    \hfill
    \begin{minipage}{0.32\textwidth}
        \centering
        \includegraphics[width=\linewidth]{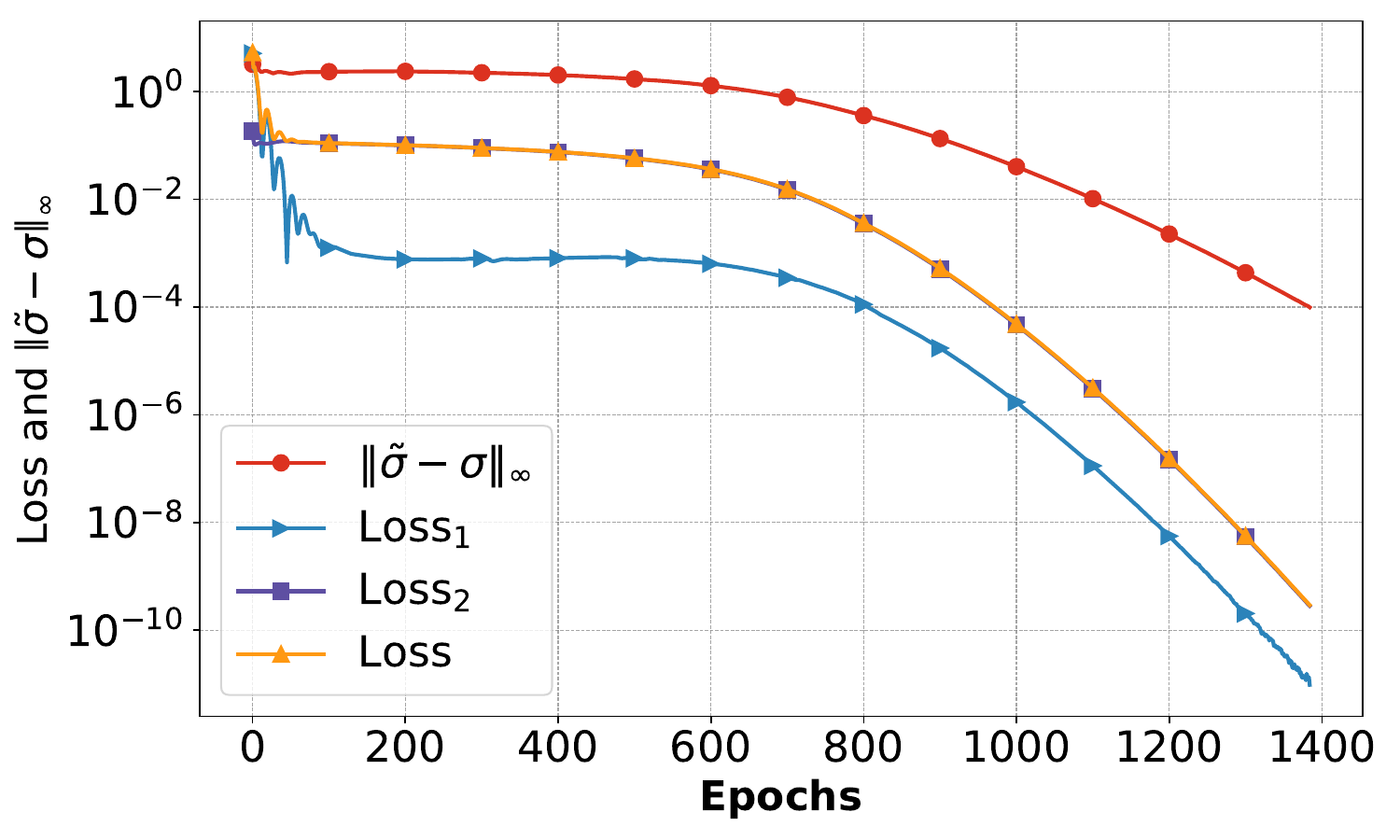}
        \vspace{-10pt}
        \subcaption{Strategy 2: Loss and $\|\tilde{\sigma}-\sigma\|_{\infty}$.}  
      \end{minipage}
    \hfill
    \begin{minipage}{0.32\textwidth}
        \centering
        \includegraphics[width=\linewidth]{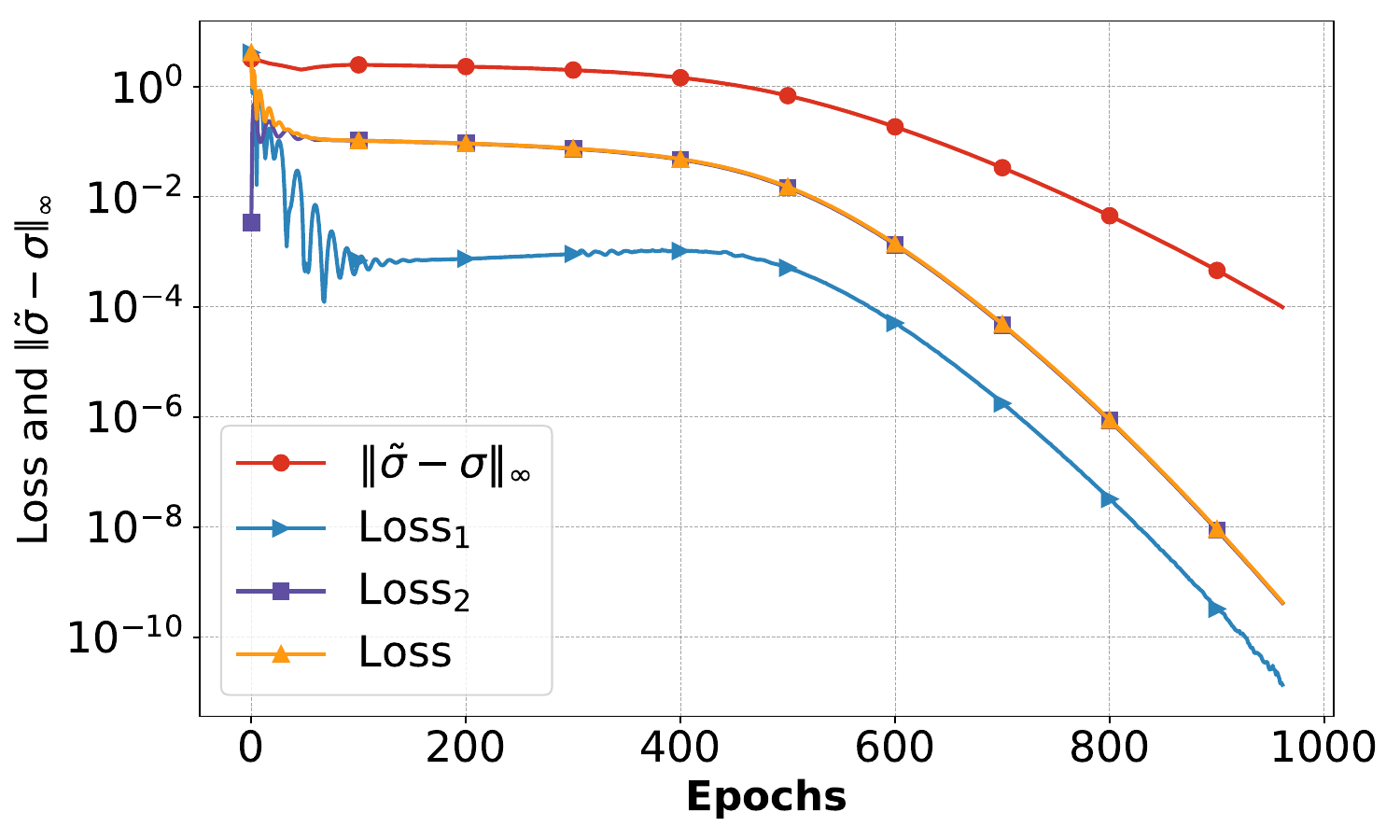}
        \vspace{-10pt}
        \subcaption{Strategy 3: Loss and $\|\tilde{\sigma}-\sigma\|_{\infty}$.}  
      \end{minipage}
  
    \vspace{0.5cm}

    \begin{minipage}{0.32\textwidth}
        \centering
        \includegraphics[width=\linewidth]{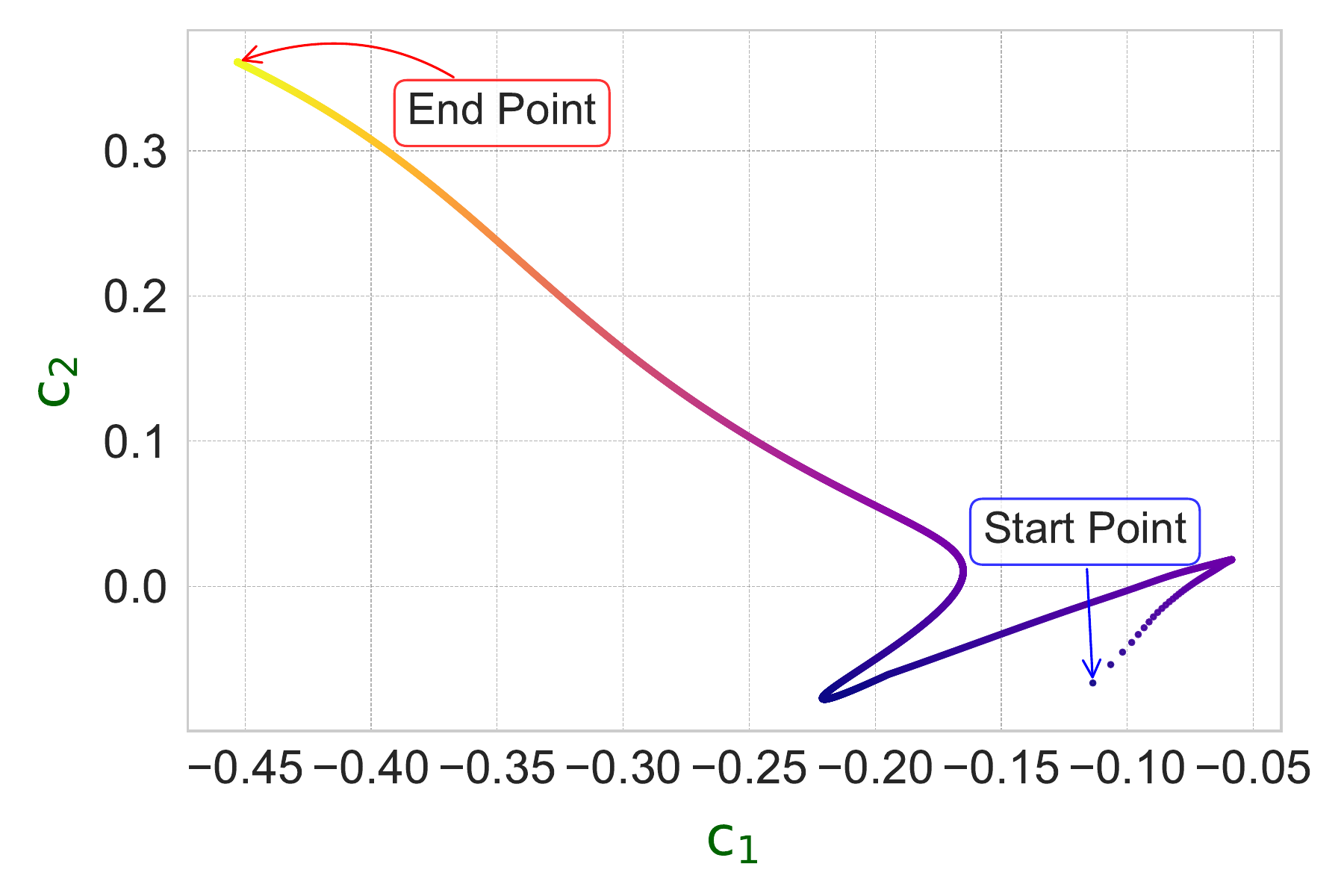}   
        \subcaption{Strategy 1 : The variation of $\textbf{c}$ during the training process.}  
      \end{minipage}
    \hfill
    \begin{minipage}{0.32\textwidth}
        \centering
        \includegraphics[width=\linewidth]{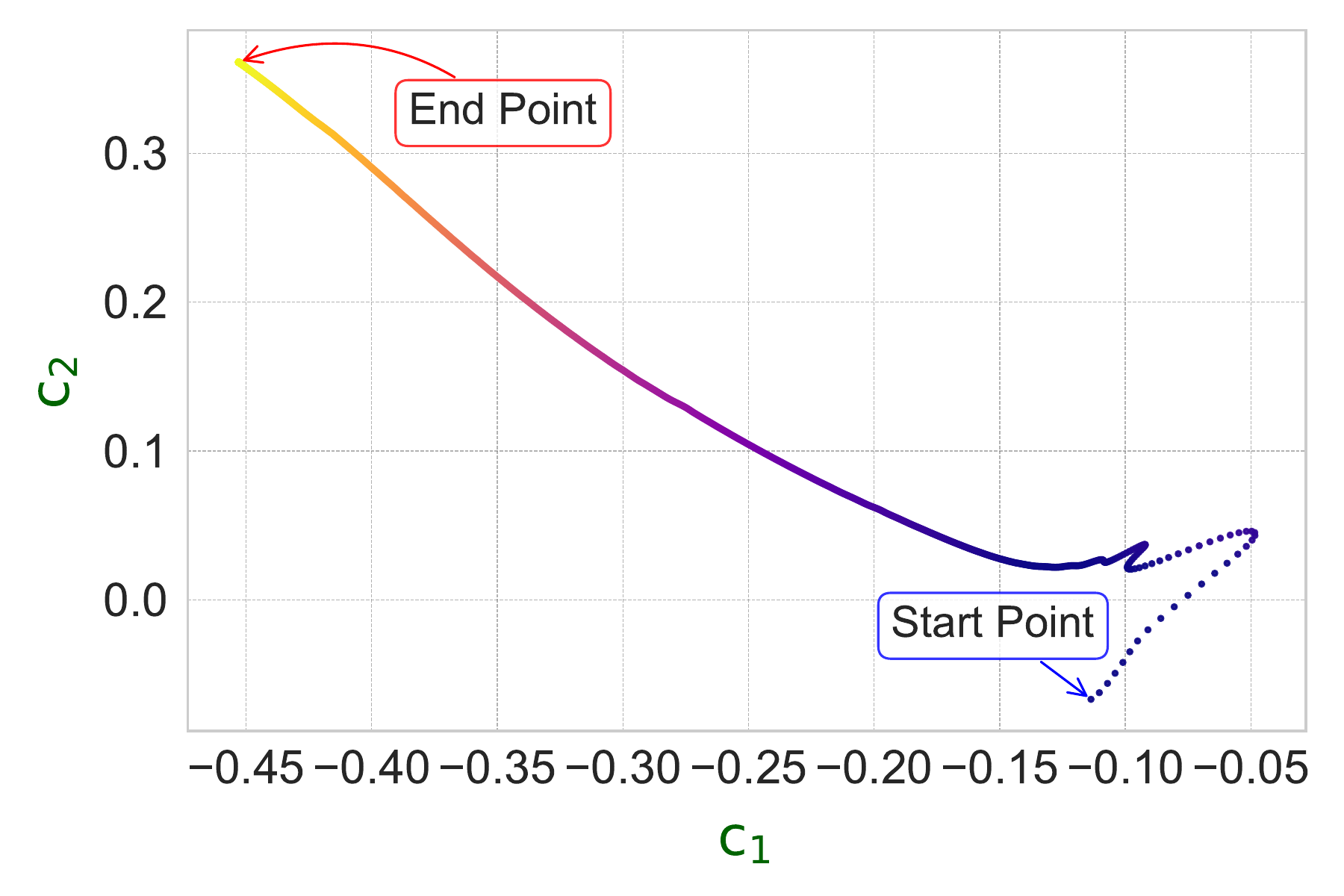}   
        \subcaption{Strategy 2 : The variation of $\textbf{c}$ during the training process.}  
      \end{minipage}
    \hfill
    \begin{minipage}{0.32\textwidth}
      \centering
      \includegraphics[width=\linewidth]{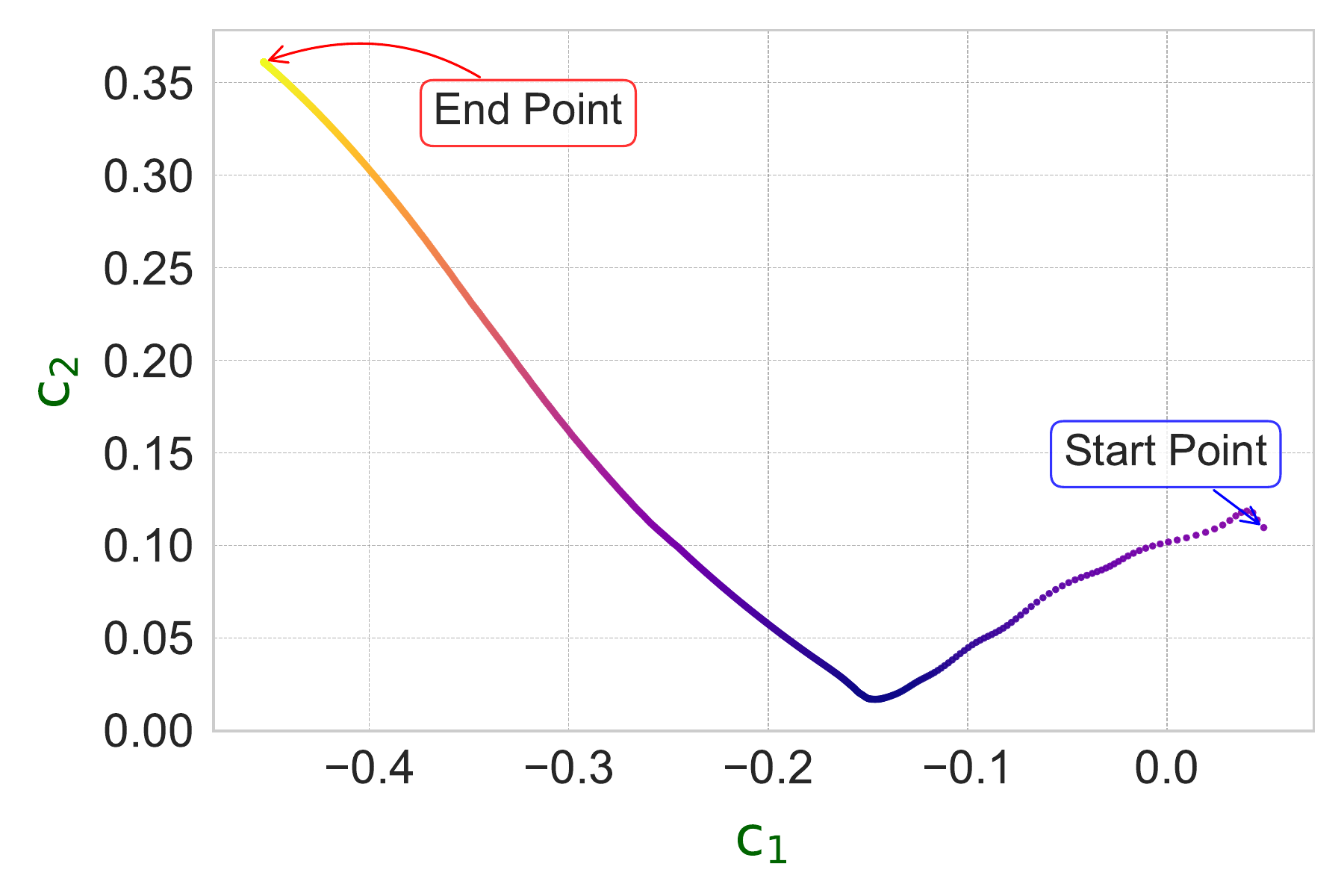}   
      \subcaption{Strategy 3 : The variation of $\textbf{c}$ during the training process.}  
    \end{minipage}
  
    \caption{Example 4.1: The Loss curve and the variation of \( \textbf{c} \) during the training process of the P-SMLP.}
    \label{Example41}
\end{figure}

\subsection{Asymmetric PGIEP}
This example is taken from \cite{dai2015solvability}.
Let
\begin{equation*}
    \begin{aligned}
        A_0=\left(\begin{array}{cc}
            1.25 & 1 \\
            1 & 1.25
            \end{array}\right), \quad B_0=\left(\begin{array}{cc}
                -0.25 & 0 \\
                0 & 0.75
                \end{array}\right),
    \end{aligned}
\end{equation*}
\begin{equation*}
    \begin{aligned}
        A_1=\left(\begin{array}{cc}
            1.9 & 0.7 \\
            0.7 & 1.7
            \end{array}\right),\quad \quad B_1=\left(\begin{array}{cc}
                0.3 & 0.3 \\
                0.3 & 0.6
                \end{array}\right)
    \end{aligned}
\end{equation*}
\begin{equation*}
    \begin{aligned}
        A_2=\left(\begin{array}{cc}
            0.575 & -0.1 \\
            -0.775 & 4
            \end{array}\right),\quad B_2=\left(\begin{array}{ll}
                -0.475 & 0 \\
                -0.225 & 0
                \end{array}\right),
    \end{aligned}
\end{equation*} 
\begin{equation*}
    A(\textbf{c})=A_0+\sum_{i=1}^{2}c_iA_i,\quad B(\textbf{c})=B_0+\sum_{i=1}^{2}c_iB_i,
\end{equation*}
\begin{equation*}
    \tilde{\sigma}=\{-1,3\}.
\end{equation*}
The parameter vector $\textbf{c}=\{c_1,c_2\}$ is to be determined.
In this experiment, a neural network with two hidden layers is employed, each containing 40 neurons.
The maximum epochs number is set to be 10000. The results are displayed in Table \ref{Example1}.
To further demonstrate the training efficiency of different strategies, we also present the variation of the loss with respect to the epochs when \(  \|\tilde{\sigma} - \sigma\|_{\infty} < 1\times 10^{-4} \), as shown in Figure \ref{Example41}.

\textbf{Analysis.}
In Table \ref{Example1}, we present the \( L^{\infty} \) errors of numerical solutions for different strategies. It is observed that the P-SMLP method yields a satisfactory solution, as the maximum norm error can approach machine precision. Furthermore, it is evident that all the strategies we proposed are effective, further demonstrating the robustness of the P-SMLP method and the reliability of the proposed model.

As illustrated in Figure \ref{Example41}. (a)-(c), when $\|\tilde{\sigma}-\sigma\|_{\infty}$ achieving the same accuracy of \(1 \times 10^{-4}\), Strategy 3 requires the fewest epochs. 
As shown in Figure \ref{Example41}. (d)-(e), the algorithm converges to the same solution under different decomposition strategies, although the convergence trajectories differ. The "start point" represents the solution \( \textbf{c} \) obtained after the first training of the P-SMLP method, while the "end point" represents the solution obtained when the algorithm has converged.

\subsection{Symmetric PGIEP}
Consider an example in \cite{dai1999algorithm}. Let
$A_0=\operatorname{diag}\left\{9,11,10,8,14\right\}$,  $B_0=\operatorname{diag}\left\{11,13,15,11,10\right\}$,  $A_1=B_1=I_5$,
\begin{equation*}
    \begin{aligned}
        A_2=\left(\begin{array}{ccccc}
            0 & 2 & 0 & 0 & 0 \\
            2 & 0 & 1 & 0 & 0 \\
            0 & 1 & 0 & 1 & 0 \\
            0 & 0 & 1 & 0 & 1 \\
            0 & 0 & 0 & 1 & 0
            \end{array}\right), \quad
            A_3=\left(\begin{array}{ccccc}
                0 & 0 & 3 & 0 & 0 \\
                0 & 0 & 0 & 2 & 0 \\
                3 & 0 & 0 & 0 & -1 \\
                0 & 2 & 0 & 0 & 1 \\
                0 & 0 & -1 & 0 & 0
                \end{array}\right), \quad
            A_4=\left(\begin{array}{ccccc}
                    0 & 0 & 0 & 1 & 0 \\
                    0 & 0 & 0 & 0 & 1 \\
                    0 & 0 & 0 & 0 & 0 \\
                    1 & 0 & 0 & 0 & 0 \\
                    0 & 1 & 0 & 0 & 0
            \end{array}\right),
    \end{aligned}
\end{equation*}
\begin{equation*}
    \begin{aligned}
            B_2=\left(\begin{array}{ccccc}
                0 & 1 & 0 & 0 & 0 \\
                1 & 0 & 1 & 0 & 0 \\
                0 & 1 & 0 & -1 & 0 \\
                0 & 0 & -1 & 0 & -1 \\
                0 & 0 & 0 & -1 & 0
                \end{array}\right), 
        \quad B_3=\left(\begin{array}{ccccc}
        0 & 0 & -1 & 0 & 0 \\
        0 & 0 & 0 & -1 & 0 \\
        -1 & 0 & 0 & 0 & 1 \\
        0 & -1 & 0 & 0 & 0 \\
        0 & 0 & 1 & 0 & 0
        \end{array}\right),  
        \quad B_4=\left(\begin{array}{ccccc}
        0 & 0 & 0 & 2 & 0 \\
        0 & 0 & 0 & 0 & 1 \\
        0 & 0 & 0 & 0 & 0 \\
        2 & 0 & 0 & 0 & 0 \\
        0 & 1 & 0 & 0 & 0
        \end{array}\right),
    \end{aligned}
\end{equation*}
\begin{equation*}
\begin{aligned}
&A_5=\left(\begin{array}{ccccc}
0 & 0 & 0 & 0 & 1 \\
0 & 0 & 0 & 0 & 0 \\
0 & 0 & 0 & 0 & 0 \\
0 & 0 & 0 & 0 & 0 \\
1 & 0 & 0 & 0 & 0
\end{array}\right)=B_5, 
 \quad A(\textbf{c})=A_0+\sum_{i=1}^5 c_i A_i, \quad B(\textbf{c})=B_0+\sum_{i=1}^5 c_i B_i,
\end{aligned}
\end{equation*}
\begin{equation*}
\tilde{\sigma}=\{0.4327,0.6636,0.9438,1.1092,1.4923\} .
\end{equation*}
The parameter vector $\textbf{c}=\{c_1,c_2,c_3,c_4,c_5\}$ is to be determined.
In this experiment, a neural network with two hidden layers is employed, each containing 40 neurons.
The maximum epochs number is set to be 200000. The results are displayed in Table \ref{Example2}.

\textbf{Analysis.}
The numerical results in Table \ref{Example2} indicate that the proposed P-SMLP method delivers ideal performance across various strategies, with Strategies 1 and 2 significantly outperforming Strategy 3 in terms of accuracy, when trained for the same number of epochs.

\begin{table}[h]
    \centering
    \caption{Example 4.2: Numerical results for the parameter vector \( \textbf{c} \) and  \( \|\tilde{\sigma} - \sigma\|_{\infty} \).}
    \begin{tabular}{>{\centering\arraybackslash}p{1.7cm} 
                    >{\centering\arraybackslash}p{1.8cm} 
                    >{\centering\arraybackslash}p{1.8cm} 
                    >{\centering\arraybackslash}p{1.8cm} 
                    >{\centering\arraybackslash}p{1.8cm} 
                    >{\centering\arraybackslash}p{1.8cm}
                    >{\centering\arraybackslash}p{2.5cm}}  
      \toprule
      \multirow{2}{*}{Strategy} & \multicolumn{5}{c}{ \textbf{c}} &\multirow{2}{*}{\textbf{$\|\tilde{\sigma}-\sigma\|_{\infty}$}}\\ 
      \cmidrule{2-6}
       & \textbf{$c_1$} & \textbf{$c_2$} & \textbf{$c_3$} & \textbf{$c_4$} & \textbf{$c_5$} &\\ 
      \midrule
       1 & 1.020168 & -0.987295 & 0.997568 & -0.946939 & 1.179553 & $6.393790 \times 10^{-4}$  \\ 
       2 & 0.978628 & 1.000773 & 0.999320 & 1.008117 & 0.945341 & $2.673864 \times 10^{-4}$\\ 
       3 & 0.878130 & -1.028322 & 1.001241 & -1.121183 & 0.372820 & $2.035528 \times 10^{-3}$\\ 
      \bottomrule
    \end{tabular}
    \label{Example2}
  \end{table}

\subsection{PGIEP with multiple eigenvalues}
Consider an example in \cite{dai2015solvability}. Let
\begin{equation*}
    \begin{aligned}
        A_0 & =\left(\begin{array}{ccccc}
            30 & 8.4 & -0.4 & 0 & 0 \\
            5.64 & 33.56 & 7.28 & -0.16 & 0 \\
            -0.09 & 7.61 & 13.75 & 2.72 & -0.048 \\
            0 & -0.12 & 3.85 & 4.89 & 0.624 \\
            0 & 0 & -0.03 & 0.9625 & 2.154
            \end{array}\right),\quad 
        A_1=\left(\begin{array}{ccccc}
                15 & 5 & 0 & 0 & 0 \\
                2.64 & 0.88 & 0 & 0 & 0 \\
                0.09 & 0.03 & 0 & 0 & 0 \\
                0 & 0 & 0 & 0 & 0 \\
                -0.015 & -0.005 & 0 & 0 & 0
                \end{array}\right),      
    \end{aligned}
\end{equation*}
\begin{equation*}
    \begin{aligned}
        A_2=\left(\begin{array}{ccccc}
            0 & 0 & 0 & 0 & 0 \\
            0 & 16 & 4 & 0 & 0 \\
            0 & 3.64 & 0.91 & 0 & 0 \\
            0 & 0.04 & 0.01 & 0 & 0 \\
            0 & 0 & 0 & 0 & 0
            \end{array}\right),\quad
        A_3=\left(\begin{array}{ccccc}
                0 & 0 & 0.1 & 0.025 & 0 \\
                0 & 0 & 0.02 & 0.005 & 0 \\
                0 & 0 & 6 & 1.5 & 0 \\
                0 & 0 & 1.88 & 0.47 & 0 \\
                0 & 0 & 0.01 & 0.0025 & 0
            \end{array}\right),
    \end{aligned}
\end{equation*}
\begin{equation*}
    \begin{aligned}
        A_4=\left(\begin{array}{ccccc}
            0 & 0 & 0 & -0.1 & -0.02 \\
            0 & 0 & 0 & 0.02 & 0.004 \\
            0 & 0 & 0 & 0.01 & 0.002 \\
            0 & 0 & 0 & 2 & 0.4 \\
            0 & 0 & 0 & 0.47 & 0.094
            \end{array}\right), \quad
            A_5=\left(\begin{array}{ccccc}
                0 & 0 & 0 & 0 & 0.15 \\
                0 & 0 & 0 & 0 & -0.05 \\
                0 & 0 & 0 & 0 & 0.01 \\
                0 & 0 & 0 & 0 & 0.01 \\
                0 & 0 & 0 & 0 & 1
                \end{array}\right),
    \end{aligned}
\end{equation*}
\begin{equation*}
    \begin{aligned}
        B_0 & =\left(\begin{array}{ccccc}
            13.007181 & 3.997188 & 0 & 0 & 0 \\
            2.498594 & 15.007181 & 3.997636 & 0 & 0 \\
            0 & 3.997636 & 5.007181 & 1.799363 & 0 \\
            0 & 0 & 2.498938 & 0.507181 & 0.600012 \\
            0 & 0 & 0 & 0.900024 & -0.892819
            \end{array}\right),\quad B_1=I_5, 
    \end{aligned}
\end{equation*}
\begin{equation*}
    \begin{aligned}
        B_2=\left(\begin{array}{ccccc}
            0 & 1 & 0 & 0 & 0 \\
            0.5 & 0 & 0 & 0 & 0 \\
            0 & 0 & 0 & 0 & 0 \\
            0 & 0 & 0 & 0 & 0 \\
            0 & 0 & 0 & 0 & 0
            \end{array}\right),\quad
            B_3=\left(\begin{array}{ccccc}
                0 & 0 & 0 & 0 & 0 \\
                0 & 0 & 1 & 0 & 0 \\
                0 & 1 & 0 & 0 & 0 \\
                0 & 0 & 0 & 0 & 0 \\
                0 & 0 & 0 & 0 & 0
                \end{array}\right),
    \end{aligned}
\end{equation*}
\begin{equation*}
    \begin{aligned}
        B_4=\left(\begin{array}{ccccc}
            0 & 0 & 0 & 0 & 0 \\
            0 & 0 & 0 & 0 & 0 \\
            0 & 0 & 0 & 0.3 & 0 \\
            0 & 0 & 0.5 & 0 & 0 \\
            0 & 0 & 0 & 0 & 0
            \end{array}\right),\quad  
            B_5=\left(\begin{array}{ccccc}
                0 & 0 & 0 & 0 & 0 \\
                0 & 0 & 0 & 0 & 0 \\
                0 & 0 & 0 & 0 & 0 \\
                0 & 0 & 0 & 0 & 0.1 \\
                0 & 0 & 0 & 0.2 & 0
                \end{array}\right), 
    \end{aligned}
\end{equation*}
\begin{equation*}
A(\textbf{c})=A_0+\sum_{i=1}^5 c_i A_i, \quad B(\textbf{c})=B_0+\sum_{i=1}^5 c_i B_i,     
\end{equation*}
\begin{equation*}
\tilde{\sigma}=\{0.5,0.5,2,3,4\}.
\end{equation*}
The parameter vector $\textbf{c}=\{c_1,c_2,c_3,c_4,c_5\}$ is to be determined.
In this experiment, a neural network with two hidden layers is employed, each containing 60 neurons. The learning rate is set to 0.01.
The maximum number of epochs is set to 200000.  The results are displayed in Table \ref{Example3}.

\begin{table}[h]
    \centering
    \caption{Example 4.3: Numerical results for the parameter vector \( \textbf{c} \) and  \( \|\tilde{\sigma} - \sigma\|_{\infty} \).}
    \begin{tabular}{>{\centering\arraybackslash}p{1.7cm} 
                    >{\centering\arraybackslash}p{1.8cm} 
                    >{\centering\arraybackslash}p{1.8cm} 
                    >{\centering\arraybackslash}p{1.8cm} 
                    >{\centering\arraybackslash}p{1.8cm} 
                    >{\centering\arraybackslash}p{1.8cm}
                    >{\centering\arraybackslash}p{2.5cm}}  
      \toprule
      \multirow{2}{*}{Strategy} & \multicolumn{5}{c}{ \textbf{c}} &\multirow{2}{*}{\textbf{$\|\tilde{\sigma}-\sigma\|_{\infty}$}}\\ 
      \cmidrule{2-6}
       & \textbf{$c_1$} & \textbf{$c_2$} & \textbf{$c_3$} & \textbf{$c_4$} & \textbf{$c_5$} &\\ 
      \midrule
       1 &  0.753495 & -0.073002 & 0.751434 & -1.872000 & -2.292557  & $1.363754\times 10^{-4}$ \\ 
       2 &  0.751432 & -0.072381 & 0.721791 & -1.533055 & -2.083673  & $1.505375\times 10^{-3}$ \\ 
       3 &  -1.628540 & 0.412606 & -0.459883 & -3.234722 & -3.431235  &  $5.006790\times 10^{-6}$\\ 
      \bottomrule
    \end{tabular}
    \label{Example3}
  \end{table}

\textbf{Analysis.}
Based on the results in Table \ref{Example3}, the P-SMLP method demonstrates a remarkable ability to achieve a satisfactory solution, even in cases with multiple eigenvalues. The performance of Strategy 3 surpasses that of the other two in terms of accuracy.
The P-SMLP method, using the orthogonal matrix generation approach of Strategy 3, solves the model proposed in this paper, achieving a high-precision result of \(10^{-6}\) within the given maximum number of epochs.

\begin{table}[h]
  \centering
  \caption{Example 4.4: Network settings and numerical results of $\|\tilde{\sigma}-\sigma\|_{\infty}$.}
  \begin{tabular}{>{\centering\arraybackslash}p{1.5cm} 
                  >{\centering\arraybackslash}p{3.0cm} 
                  >{\centering\arraybackslash}p{2.6cm} 
                  >{\centering\arraybackslash}p{2.5cm} 
                  >{\centering\arraybackslash}p{4.0cm}}  
  \toprule
   {n} & {Hidden Layer}& {Epochs} & {Strategy} & \textbf{$\|\tilde{\sigma}-\sigma\|_{\infty}$} \\ 
   \midrule 
       \multirow{3}{*}{10} & \multirow{3}{*}{[40, 40]} &\multirow{3}{*}{200000} & 1 & $7.367134 \times 10^{-5}$ \\ 
       \cmidrule(lr){4-5}
     & & &  2 & $3.397465 \times 10^{-6}$ \\   
            \cmidrule(lr){4-5}
      & & &  3 & $8.344650 \times 10^{-7}$\\   
         \midrule 
    \multirow{3}{*}{20} & \multirow{3}{*}{[60, 60]} &\multirow{3}{*}{500000} & 1 & $7.689238 \times 10^{-3}$ \\ 
       \cmidrule(lr){4-5}
     & & &  2 & $7.045269 \times 10^{-5}$ \\   
            \cmidrule(lr){4-5}
      & & &  3 & $3.933907 \times 10^{-6}$\\   
         \midrule 
    \multirow{3}{*}{30} & \multirow{3}{*}{[100, 100]} &\multirow{3}{*}{1000000} & 1 &  $ 4.254089\times 10^{-2}$ \\ 
       \cmidrule(lr){4-5}
     & & &  2 & $ 1.766682\times 10^{-4}$\\   
            \cmidrule(lr){4-5}
      & & &  3 &$ 4.053116\times 10^{-6}$ \\   
            \midrule 
          \multirow{3}{*}{40} & \multirow{3}{*}{[140, 140]} &\multirow{3}{*}{1000000} & 1 &  $ 3.458531\times 10^{-1}$ \\ 
       \cmidrule(lr){4-5}
     & & &  2 & $ 8.706093\times 10^{-3}$\\   
            \cmidrule(lr){4-5}
      & & &  3 &$ 1.457930\times 10^{-4}$ \\  
   \bottomrule
  \end{tabular}
  \label{4.Example}
\end{table}
\subsection{Large-scale PGIEP}
In this section, a large-scale PGIEP is considered. The aim is to demonstrate the effectiveness of the proposed model and the P-SMLP method as the problem dimensionality increases.

Consider an example in \cite{dalvand2021newton}. 
Let $A_0=B_0=\textbf{0}$, and define the \(n \times n\) matrices 
\begin{equation*}
    A_k = \left\{a_{i j}^k\right\}, \quad B_k=\left\{b_{i j}^k\right\},
\end{equation*}
where 
\begin{equation*}
    a_{ij}^k = \begin{cases}
1, & \text{if } \ |i-j| = k-1, \\
0, & \text{otherwise} , 
\end{cases}
\end{equation*}
and 
\begin{equation*}
    b_{i j}^k=\begin{cases}
1, \text { if } \ i=j=k, \\
0, \text { otherwise },
\end{cases}
\end{equation*}
with $k,i,j\in\{1,2,\dots,n\}$.
Let \(\mathbf{c}^* = (2, 1, 1, \dots, 1)^\top\) be an \(n\)-dimensional vector. We define the matrices
\begin{equation*}
A(\mathbf{c}^*) = A_0 + \sum_{k=1}^n c_k^* A_k, \quad B(\mathbf{c}^*) = B_0 + \sum_{k=1}^n c_k^* B_k,
\end{equation*} 
and regard the spectrum \(\tilde{\sigma}\) of the generalized eigenvalue problem  
\begin{equation*}
A(\mathbf{c}^*)x = \lambda B(\mathbf{c}^*)x
\end{equation*} 
as the known spectrum. The goal is to determine a parameter vector \(\mathbf{c}\) such that the generalized eigenvalue problem  
\begin{equation*}
(A_0 + \sum_{k=1}^n c_k A_k)x = \lambda (B_0 + \sum_{k=1}^n c_k B_k)x
\end{equation*} 
has the spectrum \(\tilde{\sigma}\).  
Four sets of experiments were conducted using matrices of varying dimensions. The neural network architecture configurations and the corresponding numerical results are summarized in Table \ref{4.Example}.

\textbf{Analysis.}
As shown in Table \ref{4.Example}, increasing the problem dimensionality requires a larger network and more training epochs. Under the same experimental settings, different decomposition strategies yield different values of $\|\tilde{\sigma} - \sigma\|_{\infty}$. Among them, Strategy 3 performs better than Strategies 1 and 2 in terms of numerical accuracy. When $n = 40$, Strategy 3 achieves an accuracy on the order of $10^{-4}$. This indicates that the proposed method and model remain effective even as the problem dimensionality increases.

\begin{table}[h]
  \centering
  \caption{Example 4.5: Accuracy of eigenvalue recovery and parameter estimation for various strategies.
  Here, $\lambda_1$, $\lambda_2$, and $\lambda_3$ denote the computed generalized eigenvalues corresponding to  $A(\textbf{c})x=\lambda B(\textbf{c})x$ by P-SMLP. The symbol \textit{eps} denotes machine precision, \textit{det} represents the matrix determinant,and \textit{cond} stands for the matrix condition number.} 
  
  \begin{tabular}{>{\centering\arraybackslash}p{1.2cm} 
                  >{\centering\arraybackslash}p{2.2cm} 
                  >{\centering\arraybackslash}p{2.2cm} 
                  >{\centering\arraybackslash}p{1.0cm} 
                  >{\centering\arraybackslash}p{2.2cm} 
                  >{\centering\arraybackslash}p{2.2cm} 
                  >{\centering\arraybackslash}p{2.2cm}}  
    \toprule
    Strategy & \textbf{$|\lambda_1-(-1)|$} & \textbf{$|\lambda_2-0.5|$} & \textbf{$\lambda_3$} & \textbf{$\|\textbf{c}-\textbf{c}^*\|_{\infty}$} & det(B(\textbf{c}))& cond(B(\textbf{c}))\\ 
    \midrule
     1 & $2.384185 \times 10^{-7}$ & $2.980232 \times 10^{-8}$ & $ \infty$ & $1.192093\times 10^{-7}$  &$\text{eps}$ &$1.415333\times 10^{17}$\\ 
     
     2 & $1.192093\times 10^{-7}$ & $1.490116\times 10^{-8}$ & $ \infty$ & $5.960464\times 10^{-8}$ &$\text{eps}$ &$6.399412\times 10^{16}$\\ 
     
     3 & $2.384185\times 10^{-7}$  & $2.980232\times 10^{-8}$ & $ \infty$ &  $1.192093\times 10^{-7}$ & $\text{eps}$ &$1.415333\times 10^{17}$\\ 
    \bottomrule
  \end{tabular}
  \label{Example5}
\end{table}

\begin{table}[h]
  \centering
  \caption{Example 4.5 : Numerical results for the parameter vector $\mathbf{c}=\{c_1,c_2,c_3\}$.} 
  
  \begin{tabular}{>{\centering\arraybackslash}p{1.2cm} 
                  >{\centering\arraybackslash}p{4.2cm} 
                  >{\centering\arraybackslash}p{4.2cm} 
                  >{\centering\arraybackslash}p{2.2cm}}  
    \toprule
    Strategy & $c_1$ & $c_2$ & $c_3$ \\ 
    \midrule
     1 & $0.9999998807907104$ & $0.9999998807907104$ & $1$ \\ 
     
     2 & $0.9999999403953552$ & $0.9999999403953552$ & $1$ \\ 
     
     3 & $0.9999998807907104$ & $0.9999998807907104$ & $1$ \\ 
    \bottomrule
  \end{tabular}
  \label{Example52}
\end{table}
\subsection{PGIEP with singular $B(\textbf{c})$ }
Let $A_0=\textbf{0}$, $B_0=I_3$,
\begin{equation*}
    \begin{aligned}
        A_1=\left(\begin{array}{ccc}
            1 & 0 & 0\\
            0 & 0 & 0\\
             0 & 0 & 0\\
            \end{array}\right), \quad A_2=\left(\begin{array}{ccc}
            0 & 1 & 0\\
            1 & 0 & 0\\
             0 & 0 & 0\\
                \end{array}\right),\quad A_3=\left(\begin{array}{ccc}
            0 & 0 & 0\\
            0 & 0 & 0\\
             0 & 0 & 1\\
                \end{array}\right),
    \end{aligned}
\end{equation*}

\begin{equation*}
    \begin{aligned}
        B_1=\left(\begin{array}{ccc}
            0 & 1 & 0\\
            1 & 0 & 0\\
             0 & 0 & 0
            \end{array}\right),\quad B_2=\left(\begin{array}{ccc}
            0 & 0 & 0\\
            0 & 1 & 1\\
             0 & 1 & 0
                \end{array}\right),\quad B_3=\left(\begin{array}{ccc}
            0 & 0 & 1\\
            0 & 0 & 0\\
             1 & 0 & 0
                \end{array}\right).
    \end{aligned}
\end{equation*}
Let $\textbf{c}^* = (1, 1, 1)^{\top}$, then
\begin{equation*}
B(\textbf{c}^*)=\left(\begin{array}{ccc}
            1 & 1 & 1\\
            1 & 2 & 1\\
             1 & 1 & 1\\
\end{array}\right)
\end{equation*}
is a singular matrix. 
In this case, the generalized eigenvalue problem $A(\mathbf{c}^*) x= \lambda B(\mathbf{c}^*)x$ has one infinite eigenvalue, and the spectrum is
\begin{equation*}
\tilde{\sigma} = \{-1, 0.5,\infty\}.
\end{equation*}
The goal is to find the parameter vector $\textbf{c}=\{c_1,c_2,c_3\}$ such that 
$A(\mathbf{c}) x= \lambda B(\mathbf{c})x$ has the spectrum $\tilde{\sigma}$.
In model (\ref{model2}), we set $\hat{P} = P + \mathbf{e}_3 \mathbf{e}_3^\top$.
In this experiment, a neural network with two hidden layers is employed, each containing 40 neurons.
The maximum number of epochs is set to 10000. The results are displayed in Tables \ref{Example5} and \ref{Example52}.

\textbf{Analysis.}
In this experiment, we investigate the performance of the proposed method in solving the PGIEP where $B(\mathbf{c}^*)$ is a singular matrix. 
Tables  \ref{Example5} and \ref{Example52} report the accuracy of eigenvalue recovery and parameter estimation under three different strategies. 
 All three strategies yield high-precision recovery of the finite eigenvalues, reaching an accuracy of $10^{-7}$. Moreover, $\lambda_3$ is correctly  identified as infinite.
The error $\|\mathbf{c} - \mathbf{c}^*\|_\infty$ is below $10^{-7}$, indicating high parameter estimation accuracy. The near-zero determinant of $B(\mathbf{c})$ confirms its singularity, which is further corroborated by a condition number exceeding $10^{16}$.
These results confirm the effectiveness and numerical stability of the proposed P-SMLP method in handling PGIEP with the singular matrix.

\section{Conclusion}\label{5}
In this paper, a novel P-SMLP is proposed for solving the model on product manifolds of the PGIEP, combining the advantages of orthogonal optimization methods and neural networks.
Unlike the MLP framework, P-SMLP provides an efficient solution for ensuring that the output naturally satisfies the orthogonality constraint, achieved through matrix decomposition techniques to enforce this hard constraint.
The main advantage of the proposed model is that it makes orthogonal optimization strategies a feasible method for PGIEPs, which is not achievable with existing models. This model is applicable to symmetric, asymmetric, and multiple-eigenvalue PGIEPs, demonstrating its universal applicability. Notably, it can also be used to solve PGIEPs where $B(\textbf{c})$ is singular.
The computational results of sufficient numerical experiments demonstrate that the proposed method can effectively solve PGIEPs.
Given these advantages, future work will focus on extending our method to address other IEP-related problems.

\section*{Acknowledgements}
This work is supported in part by the Education Department of Jilin Province (No. JJKH20250297BS).

\bibliographystyle{cas-refs}
\bibliography{cas-refs}

\end{document}